\documentclass[12pt]{amsart}
\usepackage{amsmath}
\usepackage{amssymb}
\usepackage[all]{xy}
\usepackage{longtable}
\usepackage[osf,sc]{mathpazo}
\usepackage{euscript}
\usepackage{calrsfs}
\usepackage{array}

\newtheorem{theorem}[equation]{Theorem}
\newtheorem{lemma}[equation]{Lemma}
\newtheorem{proposition}[equation]{Proposition}
\newtheorem{corollary}[equation]{Corollary}
\newtheorem{conjecture}[equation]{Conjecture}

\theoremstyle{remark}
\newtheorem{remark}[equation]{Remark}

\numberwithin{equation}{subsection}

\newcommand{\FF}{\mathbb{F}}
\newcommand{\ZZ}{\mathbb{Z}}
\newcommand{\QQ}{\mathbb{Q}}
\newcommand{\RR}{\mathbb{R}}

\newcommand{\TT}{\mathbb{T}}
\newcommand{\GG}{\mathbb{G}}

\newcommand{\CC}{\mathbb{C}}

\newcommand{\NN}{\mathbb{N}}

\newcommand{\bu}{\mathbf{u}}
\newcommand{\bv}{\mathbf{v}}

\newcommand{\bC}{\mathbf{C}}

\newcommand{\bz}{\mathbf{z}}

\newcommand{\cL}{\mathcal{L}}
\newcommand{\cF}{\mathcal{F}}

\newcommand{\cV}{\mathcal{V}}

\DeclareMathAlphabet{\matheur}{U}{eur}{m}{n}

\newcommand{\fs}{\mathfrak{s}}

 \DeclareMathOperator{\Lie}{Lie}
\DeclareMathOperator{\Ker}{Ker} \DeclareMathOperator{\GL}{GL}
\DeclareMathOperator{\Mat}{Mat} 
\DeclareMathOperator{\End}{End}

 \DeclareMathOperator{\wt}{wt}
\DeclareMathOperator{\Ext}{Ext}  

\DeclareMathOperator{\rank}{rank}

\DeclareMathOperator{\SL}{SL}

\newcommand{\ok}{\overline{k}}

\newcommand{\Sp}{\mathrm{Span}}
\newcommand{\rk}{\mathrm{rank}}

\newcommand{\tr}{\mathrm{tr}}

\newcommand{\power}[2]{{#1 [\![ #2 ]\!]}}
\newcommand{\laurent}[2]{{#1 (\!( #2 )\!)}}

\begin{document}
\title[Linear relations among double zeta values]{{\large{Linear relations among double zeta values in positive characteristic}}}

\author{Chieh-Yu Chang}

\dedicatory{Dedicated to the memory of my father}

\address{Department of Mathematics, National Tsing Hua University, Hsinchu City 30042, Taiwan
  R.O.C.}

\email{cychang@math.nthu.edu.tw}

\thanks{The author was partially supported by a Golden-Jade
  fellowship of the Kenda Foundation and MOST Grant
  102-2115-M-007-013-MY5. He thanks NCTS for offering a center scientist position, which is very helpful to research.}


\date{July 6, 2016}

\begin{abstract}
For each integer $n\geq 2$, we study linear relations among weight $n$ double zeta
values and the $n$th power of the Carlitz period over the rational function field $\FF_
{q}(\theta)$. We show that all the $\FF_q(\theta)$-linear relations are induced from the $\FF_{q}[t]
$-linear relations among certain explicitly constructed special points in the $n$th
tensor power of the Carlitz module. We then establish a principle of Siegel's lemma for
computing and determining the $\FF_{q}[t]$-linear relations mentioned above, and thus
obtain an effective criterion for computing the dimension of weight $n$ double zeta
values space.
\end{abstract}


\maketitle

\section{Introduction}
\subsection{Classical theory}
 Classical multiple zeta values (abbreviated as MZV's) are defined by the series: for $\fs=(s_{1},\ldots,s_{r})\in \NN^{r}$ with $s_{1} \geq2$,
\[\zeta(\fs):=\sum_{n_{1}>\cdots>n_{r}\geq
  1} \frac{1}{n_{1}^{s_{1}}\cdots n_{r}^{s_{r}} } \in \RR^{\times }.\]
Here ${\rm{wt}}(\fs):=\sum_{i=1}^{r}s_{i}$ is called the weight and $r$ is called the depth of the presentation $\zeta_{A}(\fs)$. MZV's have many connections with various research topics. For example, see~\cite{Z93, Z94, Cart02, Andre04, B12, Zh16}.

 Let $\mathrm{Z}_{0}:=\QQ$, $\mathrm{Z}_{1}:=\left\{ 0\right\}$ and $\mathrm{Z}_{n}$ be the $\QQ$-vector space spanned by the weight $n$ MZV's for integers $n\geq 2$. Putting $\mathrm{Z}:=\sum_{n\geq 0} \mathrm{Z}_{n}$. It is well known that $\mathrm{Z}_{m}\mathrm{Z}_{n}\subseteq \mathrm{Z}_{m+n}$, and so $\mathrm{Z}$ has a $\QQ$-algebra structure. The Goncharov's direct sum conjecture~\cite{G97} asserts that $\mathrm{Z}$ is a graded algebra (graded by weights). Therefore, understanding the $\QQ$-algebraic relations among MZV's boils down to understanding the $\QQ$-linear relations among the same weight MZV's. However, to date computing the dimension of $\mathrm{Z}_{n}$ for each $n$ is out of reach. Note that Zagier's dimension conjecture predicts that $d_{n}:=\dim_{\QQ}\mathrm{Z}_{n}$ satisfies the recursive relation
 \[  d_{n}=d_{n-2}+d_{n-3}\hbox{ for }n\geq3,  \]
and one knows by Goncharov and Terasoma that $\dim_{\QQ}\mathrm{Z}_{n}\leq d_{n}$ for each $n$ (see \cite{Te02, DG05}).

Now, we focus on depth two MZV's, which are called double zeta values. It is a natural question to ask how to compute the dimension of the $\QQ$-vector space
\[ \mathrm{DZ}_{n}:= \Sp_{\QQ} \left\{(2\pi \sqrt{-1})^{n},\zeta(2,n-2),\zeta(3,n-3),\cdots,\zeta(n-1,1)\right\}
\] for $n\geq 3$. This is still a very difficult problem in the classical theory as $\dim_{\QQ}\mathrm{DZ}_{n} $ is only known for $n=3,4$. Zagier (cf.~\cite{Z94}) gave a conjectural formula for $d_{n}$ in terms of the dimension of weight $n$ cusp forms.
\begin{conjecture}{\rm{(Zagier)}}\label{Conj1}
For $n\geq 3$, we put
\[
s_{n}:=\begin{cases}
\frac{n}{2}-1-\dim_{\CC}S_{n}({\rm{SL}}_{2}(\ZZ))& \hbox{ if }n\hbox{ is even} , \\
\frac{n+1}{2} &\hbox{ if }$n$\hbox{ is odd}, \\
\end{cases}
\] where $S_{n}(\SL_{2}(\ZZ))$ is space of weight $n$ cusp forms for $\SL_{2}(\ZZ)$. Then we have \[\dim_{\QQ}\mathrm{DZ}_{n}=s_{n}.\]
\end{conjecture}
The best known result toward this conjecture is due to Gangl, Kaneko and Zagier~\cite{GKZ06}, who showed that $\dim_{\QQ}\mathrm{DZ}_{n}\leq s_{n}$ for each $n\geq 3$. One of their approaches is to introduce and study the double Eisenstein series to explain the relation between double zeta values and cusp forms for $\SL_{2}(\ZZ)$. The main result of this paper is to establish an effective criterion for computing the analogue of $\dim_{\QQ}\mathrm{DZ}_{n}$ in the positive characteristic function field setting.

\subsection{The main result}
Let $A:=\FF_{q}[\theta]$ be the polynomial ring in the variable $\theta$ over the finite field $\FF_{q}$ of $q$ elements with characteristic $p$, and $k$ be its quotient field. Denote by $\infty$ the infinite place of $k$. Let $k_{\infty}:= \laurent{\FF_{q}}{\frac{1}{\theta}}$ be the $\infty$-adic completion of $k$, and $\overline{k_{\infty}}$ be a fixed algebraic closure of $k_{\infty}$. Denote by $\CC_{\infty}$ the $\infty$-adic completion of $\overline{k_{\infty}}$. Finally, we let $A_{+}$ be the set of all monic polynomials in $A$. We then have the following comparisons:
\[ A_{+}\leftrightarrow \NN, \hbox{ } A\leftrightarrow \ZZ,\hbox{ } k\leftrightarrow \QQ, \hbox{ } k_{\infty}\leftrightarrow \RR, \hbox{ } \CC_{\infty}\leftrightarrow \CC     .\]

For any $r$-tuple of positive integers $\fs=(s_{1},\ldots,s_{r})\in \NN^{r}$, Thakur~\cite{T04} introduced the multiple series
\[
  \zeta_{A}(\fs) := \sum \frac{1}{a_1^{s_1} \cdots
  a_r^{s_r}} \in k_{\infty},
\]
where the sum is taken over $r$-tuples $(a_1, \dots, a_r)\in A_{+}^{r}$ satisfying the strict inequalities $\deg_{\theta} a_1 > \cdots > \deg_{\theta} a_r$. In analogy with the classical MZV's, $r$  is called the depth and $\rm{wt}(\fs):=s_1+\cdots+s_r$ is called the weight of the presentation $\zeta_{A}(\fs)$. These special values $\zeta_{A}(\fs)$ are called multizeta values (abbreviated as MZV's too) and each of them is non-vanishing by \cite{T09a}. Moreover, these MZV's have a $t$-motivic interpretation in the sense that they occur as periods of certain mixed Carlitz-Tate $t$-motives by the work of Anderson and Thakur~\cite{AT09}.

In \cite{T10}, Thakur showed that the product of two MZV's can be expressed as an $\FF_{p}$-linear combination of some MZV's with the same weight, which is regarded as a kind of shuffle product relation (cf.~\eqref{E:Chen}), and so the $\FF_{p}$-vector space spanned by MZV's has a ring structure. Further, an analogue of Goncharov's direct sum conjecture was shown by the author~\cite{C14}, that the $k$-algebra generated by all MZV's is a graded algebra (graded by weights). In other words, all $k$-linear relations among MZV's are generated by those $k$-linear relations among the same weight MZV's.

In the classical theory of MZV's, (regularized) double shuffle relations give rise to rich $\QQ$-linear relations among the same weight MZV's (see~\cite{IKZ06}). Unlike the classical situation, there is no natural total order on $A_{+}$ and so far a nice analogue of the iterated integral expression for MZV's is not developed yet. To date there is no different expression for the product of two MZV's other than Thakur's relation mentioned above, and hence we do not have the analogue of double shuffle relations to produce $k$-linear relations among MZV's naturally. In his Ph.D. thesis, Todd~[To15] tried to produce $k$-linear relations among the same weight MZV's using power sums and used lattice reduction methods to give a conjecture on the dimensions in question.

Let $\tilde{\pi}$ be a fixed fundamental period of the Carlitz $\FF_{q}[t]$-module $\bC$, which plays the analogous role of $2\pi \sqrt{-1}$, and let $\bC^{\otimes n}$ be the $n$th tensor power of the Carlitz module $\bC$ for a positive integer $n$ (see \S\S~\ref{sec:t-modules} for the definitions). The main result of this paper gives a new point of view to completely determine the $k$-linear relations among weight $n$ double zeta values together with  $\tilde{\pi}^{n}$. It is stated as follows and its proof is given in Corollary~\ref{C:ReMainThm1} and Theorem~\ref{T:Siegel}.
\begin{theorem}\label{T:MainThmIntroduction}
Let $n\geq 2$ be a positive integer. Put
\[ \mathcal{V}:=\left\{ (s_{1},s_{2})\in \NN^{2}; s_{1}+s_{2}=n\hbox{ and }(q-1)|s_{2}  \right\}  .\]
\begin{enumerate}
\item For each $\fs\in \cV$, we explicitly construct a special point $\Xi_{\fs}\in \bC^{\otimes n}(A)$ so that
    \[
  \begin{array}{rl}
     & \dim_{k}\Sp_{k}\left\{\tilde{\pi}^{n},\zeta_{A}(1,n-1),\zeta_{A}(2,n-2),\cdots,\zeta_{A}(n-1,1) \right\} \\
     &\\
     =& n- \lfloor \frac{n-1}{q-1} \rfloor+\rank_{\FF_{q}[t]} \Sp_{\FF_{q}[t]}\left\{\Xi_{\fs} \right\}_{\fs\in \mathcal{V}}. \\
  \end{array}
\]
\item We establish an effective algorithm for computing the rank \[\rank_{\FF_{q}[t]} \Sp_{\FF_{q}[t]}\left\{\Xi_{\fs} \right\}_{\fs\in \mathcal{V}}.\]
\end{enumerate}
\end{theorem}

In other words, we relate the $k$-linear relations among double zeta values to  the $\FF_{q}[t]$-linear relations among the special points $\left\{ \Xi_{\fs}\right\}_{\fs\in \cV}$, and which can be effectively computed and determined.

We mention that although Todd~\cite{To15} provided some ways of producing $k$-linear relations among the same weight MZV's, it is not clear how to derive $k$-linear relations among weight $n$ double zeta values together with $\tilde{\pi}^{n}$ from Todd's relation. When the given weight $n\geq 2$ is $A$-even, ie., $(q-1)| n$ (as $q-1$ is the cardinality of the unit group $A^{\times}$), one can use the following formula to produce linear relations.  For two positive $A$-even integers $r$ and $s$ with $r+s=n$, one has
\begin{equation}\label{E:Chen}
\begin{aligned}\zeta_{A}(r) \zeta_{A}(s)&= \zeta_{A} (r,s) + \zeta_{A} (s,r) + \zeta_{A} (r+s) \\
&+  \sum_{i+j =n  \atop   (q-1) | j} \left[(-1)^{s-1}\binom{j-1}{s-1}+(-1)^{r-1}\binom{j-1}{r-1}  \right] \zeta_{A} (i, j)\ \ \ \
\end{aligned}
\end{equation}
(see \cite{T10} for the existence of such relations and \cite{Chen15} for the explicit formula). By work of Carlitz~\cite{Ca35}, we know that
\[ \zeta_{A}(r)/\tilde{\pi}^{r}\in k,\hbox{ } \zeta_{A}(s)/\tilde{\pi}^{s}\in k,\hbox{ } \zeta_{A}(n)/\tilde{\pi}^{n}\in k,     \]
and so (\ref{E:Chen}) gives rise to a nontrivial  linear relation among $\tilde{\pi}^{n}$ and weight $n$ double zeta values. As all the coefficients of the double zeta values are in $\FF_{p}$,  Thakur call it an \lq\lq $\FF_{p}$-linear reation\rq\rq.  Our effective algorithm based on Theorem~\ref{T:MainThmIntroduction} is able to find all the independent $k$-linear relations in question. As observed by Thakur~\cite{T09b}, those $\FF_{p}$-linear relations produced by (\ref{E:Chen}) can not generate all the $k$-linear relations, and our computational data can capture the difference precisely. We refer the reader to the end of this paper.

We mention that there is a difference between Conjecture~\ref{Conj1} and our results, and refer the reader to Remark~\ref{Rem:comparison} about the detailed comparison. We further mention that in \cite{Chen16}, Drinfeld double Eisenstein series are introduced. Since double zeta values occur as the constant terms of Drinfeld double Eisenstein series,  one naturally expects that Drinfeld cusps forms~\cite{Go80, Ge88} can have connections with double zeta values if they can be shown to be a subspace of the space of Drinfeld double Eisenstein series (cf.~\cite{GKZ06}).

\subsection{Methods of proof} Let notation and hypothesis be given as in Theorem~\ref{T:MainThmIntroduction}. We outline the major steps in the proof of Theorem~\ref{T:MainThmIntroduction}.
\begin{enumerate}
\item[(I)] $\underline{\hbox{A necessary condition}}$. We show that all the $k$-linear relations among the set \[\left\{ \tilde{\pi}^{n}\right\}\cup \left\{\zeta_{A}(1,n-1),\ldots,\zeta_{A}(n-1,1)\right\}\] are those coming from the $k$-linear relations among  $\left\{ \tilde{\pi}^{n}\right\}\cup \left\{ \zeta_{A}(\fs) \right\}_{\fs\in \cV}$, and so we are reduced to studying the $k$-linear relations among  $\left\{\tilde{\pi}^{n} \right\}\cup \left\{ \zeta_{A}(\fs) \right\}_{\fs\in \cV}$. This result is Corollary~\ref{C:NecCondDepth2}.
\item[(II)] $\underline{\hbox{Logarithmic interpretion}}$. Let $r\geq 2$. For any $\fs=(s_{1},\ldots,s_{r})\in \NN^{r}$ with $\zeta_{A}(s_{2},\ldots,s_{r})$ Eulerian, ie.,
\[\zeta_{A}(s_{2},\ldots,s_{r})/\tilde{\pi}^{s_{2}+\cdots+s_{r}}\in k,\] we relate $\zeta_{A}(\fs)$ to the last coordinate of the logarithm of $\bC^{\otimes (s_{1}+\cdots+s_{r})}$ at an explicit integral point. This result is  Theorem~\ref{T1:LogMZV}.
\item[(III)] $\underline{\hbox{The identity}}$. Using the results in (I) and (II), we establish the equality in Theorem~\ref{T:MainThmIntroduction}~(1) by appealing to Yu's transcendence theory~\cite{Yu91} for the last coordinate of the logarithm of $\bC^{\otimes n}$ at algebraic points. This result is Corollary~\ref{C:ReMainThm1}.
\item[(IV)] $\underline{\hbox{A Siegel's Lemma}}$. We establish a principle of Siegel's lemma for integral points in $\bC^{\otimes n}$ to achieve Theorem~\ref{T:MainThmIntroduction}~(2). This result is Theorem~\ref{T:Siegel}.
\end{enumerate}

The four steps laid out above give an approach toward producing $k$-linear relations among the same weight MZV's of arbitrary depths. Actually, for weight $n\geq 2$ combining the ideas of (II), (III) and (IV) above one determine all the $k$-linear relations among the set \[\left\{ \zeta_{A}(n)\right\}\cup \left\{ \zeta_{A}(\fs);\fs\in \mathcal{E}_{n} \right\},\]  where $\mathcal{E}_{n}$ is the set consisting of all $\fs\in \NN^{r}$ with $r\geq2$ and ${\rm{wt}}(\fs)=n$ satisfying that $\zeta_{A}(\fs')$ is Eulerian, where $\fs'=(s_{2},\ldots,s_{r})$ for $\fs=(s_{1},\ldots,s_{r})$. This result is   Corollary~\ref{C:DimHigherDepth} together with Theorem~\ref{T:Siegel}. In \cite{CPY14} an effective criterion for Eulerian MZV's is established, and conjecturally one can describe the set $\mathcal{E}_{n}$ precisely (see~\cite[\S~6.2]{CPY14}). We note that assuming Todd's dimension conjecture~\cite{To15}, for weight $n\geq 2$ the $k$-linear relations among the set  $\left\{ \zeta_{A}(n)\right\}\cup \left\{ \zeta_{A}(\fs);\fs\in \mathcal{E}_{n} \right\}$ are not enough to generate all the $k$-linear relations among weight $n$ MZV's.

The idea of proving (I) above is to construct a suitable system of Frobenius difference equations for each $k$-linear relation among the double zeta values and $\tilde{\pi}^{n}$, and then use ABP-criterion~\cite[Thm.~3.1.1]{ABP04} in the study of certain $\Ext^{1}$-modules. For the proof of (II) above, we need an explicit formula for the bottom row of each coefficient matrix of the logarithm of  ${\bC^{\otimes n}}$ due to Papanikolas~\cite{P14}. Combining with the period interpretation of MZV's given by Anderson-Thakur~\cite{AT09}, one is able to relate $\zeta_{A}(\fs)$ for those $\fs\in \cV$ in Theorem~\ref{T:MainThmIntroduction} to the last coordinate of the logarithm of $\bC^{\otimes n}$. The proof of (III) is to use the functional equation of the exponential function of $\bC^{\otimes n}$ and apply Yu's theory~\cite{Yu91}. To achieve (IV), we translate the effectiveness question to a question of the type of Siegel's lemma for certain difference equations, and we prove it directly.

\subsection{Outline of this paper} In order to let the present paper be self-contained, in \S\S~2 we give some preliminaries about some major results in \cite{AT90, Yu91, CPY14}.  We then give proofs of (I)-(IV) above in \S\S~3-6 respectively. The proof of Theorem~\ref{T:MainThmIntroduction} is given in Corollary~\ref{C:ReMainThm1} and Theorem~\ref{T:Siegel}. In \S\S~\ref{sec:algorithm}, we give an effective algorithm for implementing Theorem~\ref{T:MainThmIntroduction}, and at the end of this paper we provide some data of this computation using Magma.

\subsection*{Acknowledgements}
I am very grateful to M.~Kaneko and J.~Yu for their helpful conversations that inspire this project, and to M. Papanikolas for sharing his formula with me, and to Y.-H. Lin for writing the Magma code to compute the dimensions, and to W.~D.~Brownawell, D.~Goss and D.~Thakur for helpful comments. I further thank H.-J.~Chen,  H.~Furusho, Y.-L.~Kuan, Y.~Mishiba, K.~Tasaka, A.~Tamagawa, S.~Yasuda and J. Zhao for many useful discussions and comments. Part of this work was carried out when I visited Beijing Tsing Hua University. I particularly thank Prof. L.~Yin for his kind invitation and support during his lifetime. Finally, I would like to express my gratitude to the referee for providing many helpful comments that greatly improve the exposition of this paper.

\section{Preliminaries}
\subsection{Notation}We adopt the following notation.
\begin{longtable}{p{0.5truein}@{\hspace{5pt}$=$\hspace{5pt}}p{5truein}}
$\FF_q$ & the finite field with $q$ elements, for $q$ a power of a
prime number $p$. \\
$\theta$, $t$ & independent variables. \\
$A$ & $\FF_q[\theta]$, the polynomial ring in the variable $\theta$ over $\FF_q$.
\\
$A_{+}$ & set of monic polynomials in A.
\\
$k$ & $\FF_q(\theta)$, the fraction field of $A$.\\
$k_\infty$ & $\laurent{\FF_q}{1/\theta}$, the completion of $k$ with
respect to the infinite place $\infty$.\\
$\overline{k_\infty}$ & a fixed algebraic closure of $k_\infty$.\\
$\ok$ & the algebraic closure of $k$ in $\overline{k_\infty}$.\\
$\CC_\infty$ & the completion of $\overline{k_\infty}$ with respect to
the canonical extension of $\infty$.\\
$|\cdot|_{\infty}$& a fixed absolute value for the completed field $\CC_{\infty}$ so that $|\theta|_{\infty}=q$.\\
$\power{\CC_\infty}{t}$ & ring of formal power series in $t$ over $\CC_{\infty}$.\\
$\laurent{\CC_\infty}{t}$ & field of Laurent series in $t$ over $\CC_{\infty}$.\\
$\TT$& the ring of power series in  $\power{\CC_\infty}{t}$ convergent on the closed unit disc.\\
$\tilde{\pi}$ & a fixed fundamental period of the Carlitz module $\bC$.\\
$\GG_{a}$ & the additive group scheme over $A$.\\
\end{longtable}

\subsection{Anderson t-modules revisited}\label{sec:t-modules}
We let $\tau:=(x\mapsto x^{q})$ be the $q$th power endomorphism of $\CC_{\infty}$,  and define $\CC_{\infty}\{\tau\}$ to be the twisted polynomial ring in the variable $\tau$ over $\CC_{\infty}$ subject to the relation
\[ \tau \alpha=\alpha^{q}\tau\hbox{ for }\alpha\in \CC_{\infty}.  \] It follows that we have the matrix ring $\Mat_{n}\left( \CC_{\infty}\{\tau\}\right)$ with entries in $\CC_{\infty}\{\tau\}$ and any element in this ring can be expressed as
\[\varphi=\sum_{i\geq 0}a_{i}\tau^{i}   \]
with $a_{i}\in \Mat_{n}(\CC_{\infty})$ and $a_{i}=0$ for $i\gg 0$. We denote by $\partial \varphi:=a_{0}$, the constant matrix of $\varphi.$ For convenience, we still denote by $\tau$ the operator on $\CC_{\infty}^{n}$ which raises each component to the $q$th power.  We denote by $\GG_{a}^{n}$ the $n$-dimensional additive group scheme over $A$ and note that $\Mat_{n}(\CC_{\infty}\{\tau\})$ can be identified with $\End_{\FF_{q}}\left(\GG_{a}^{n}(\CC_{\infty})\right)$, the ring of $\FF_{q}$-linear endomorphisms of the algebraic group $\GG_{a}^{n}(\CC_{\infty})=\CC_{\infty}^{n}$. Via this identification $\partial \varphi$ is the tangent map of the morphism $\varphi:\GG_{a}^{n}(\CC_{\infty})\rightarrow \GG_{a}^{n}(\CC_{\infty})$ at the identity.

By an {\it{$n$-dimensional $t$-module}} we mean a pair $E=(\GG_{a}^{n},\phi)$, where the underlying space of $E$ is $\GG_{a}^{n}(\CC_{\infty})$, which is equipped with an $\FF_{q}[t]$-module structure via the $\FF_{q}$-linear ring homomorphism
\[ \phi: \FF_{q}[t] \rightarrow \Mat_{n}\left(\CC_{\infty}\{\tau\}\right) \]
so that $\left(\partial \phi_{t}-\theta I_{n}\right)$ is a nilpotent matrix. For such a $t$-module, Anderson~\cite{A86} showed that there is a unique $n$-variable power series $\exp_{E}$ defined on the whole $\CC_{\infty}^{n}$, called the {\it{exponential}} of the $t$-module $E$, for which:
\begin{enumerate}
\item[$\bullet$] $\exp_{E}:\Lie\left( \GG_{a}^{n}(\CC_{\infty})\right)= \CC_{\infty}^{n}\rightarrow E(\CC_{\infty}) $ is $\FF_{q}$-linear.
\item[$\bullet$] $\exp_{E}$ is of the form $ I_{n} +\sum_{i=1}^{\infty}\alpha_{i}\tau^{i}$ with $\alpha_{i}\in \Mat_{n}(\CC_{\infty})$.
\item[$\bullet$] $\exp_{E}$ satisfies the functional identity: for all $a\in\FF_{q}[t]$,
\[ \exp_{E} \circ \partial \phi_{a} =\phi_{a}\circ \exp_{E}. \]
\end{enumerate}

One typical example of a nontrivial $t$-module is the {\it{$n$th tensor power of the Carlitz module}} denoted by $\bC^{\otimes n}=(\GG_{a}^{n},[\cdot]_{n})$ for $n\in \NN$ (see~\cite{AT90}). Here $[\cdot]_{n}$ is the $\FF_{q}$-linear ring homomorphism $[\cdot]_{n}:\FF_{q}[t]\rightarrow \Mat_{n}(\CC_{\infty}[\tau])$ determined by
\[ [t]_{n}=\left(
             \begin{array}{cccc}
               \theta & 1 & \cdots & 0 \\
                & \theta & \ddots & \vdots \\
                &  & \ddots & 1 \\
                &  &  & \theta \\
             \end{array}
           \right)+ \left(
                      \begin{array}{cccc}
                        0 & 0 & \cdots & 0 \\
                        \vdots &  &  & \vdots \\
                        0 & 0 & \cdots & 0 \\
                        1 & 0 & \cdots & 0 \\
                      \end{array}
                    \right)\tau
 .\] When $n=1$, $\bC:=\bC^{\otimes 1}$ is called the Carlitz $\FF_{q}[t]$-module.

We denote by $\exp_{n}:=\exp_{\bC^{\otimes n}}$ the exponential  of $\bC^{\otimes n}$. We define $\log_{n}:=\log_{\bC^{\otimes n}}$ to be the unique power series of $n$ variables, called the {\it{logarithm}} of $\bC^{\otimes n}$, for which:
\begin{enumerate}
\item[$\bullet$] $\log_{n}$ is of the form $\log_{n}= I_{n}+\sum_{i=1}^{\infty} P_{i} \tau^{i}$ with $P_{i}\in \Mat_{n}(k)$.
\item[$\bullet$] $\log_{n}$ satisfies the functional identity: for any $a\in \FF_{q}[t]$,
\[ \log_{n}\circ [a]_{n}=\partial [a]_{n}\circ \log_{n} . \]
\end{enumerate}
As formal power series we note that $\exp_{n}$ and $\log_{n}$ are inverses of each other:
\[\exp_{n}\circ\log_{n}={\rm{identity}}=\log_{n}\circ\exp_{n}   .\]
In the case of $n=1$, $\exp_{\bC}$ and $\log_{\bC}$ are called the Carlitz exponential and Carlitz logarithm respectively, and these two functions can be written down explicitly as follows. Putting $D_{0}=1$ and $D_{i}:=\prod_{j=0}^{i-1}(\theta^{q^{i}}-\theta^{q^{j}})$ for $i\in \NN$, then
\[ \exp_{\bC}=\sum_{i=0}^{\infty} \frac{1}{D_{i}} \tau^{i} .\]
We further put $L_{0}:=1$ and $L_{i}:=(\theta-\theta^{q})\cdots (\theta-\theta^{q^{i}})$ for $i\in \NN$, and  then
\[ \log_{\bC}=\sum_{i=0}^{\infty} \frac{1}{L_{i}}\tau^{i}.  \]
For the details, see~\cite{Go96, T04}.
\subsection{Review of the theories of Anderson-Thakur and Yu}
\subsubsection{Theory of Anderson-Thakur} In their seminal paper~\cite{AT90}, Anderson and Thakur first showed that for each $n\in \NN$, $\exp_{n}:\CC_{\infty}^{n}\rightarrow \bC^{\otimes n}(\CC_{\infty})$ is surjective and its kernel is of rank one over $A$ in the sense that
\[ \Ker \exp_{n}=\partial [\FF_{q}[t]]_{n} \lambda_{n},   \]
where $\lambda_{n}\in \CC_{\infty}^{n}$ is of the form
\[\lambda_{n}=\left(
                \begin{array}{c}
                  * \\
                  \vdots \\
                  \tilde{\pi}^{n} \\
                \end{array}
              \right)
  .\] The $\tilde{\pi}$ above is a  fundamental period of the Carlitz module $\bC$ in the sense that $\Ker \exp_{\bC}=A\tilde{\pi}$, and it is fixed throughout this paper.

 For a non-negative integer $n$, we express $n$ as
\[
n=\sum_{i=0}^{\infty} n_{i}q^{i} \quad \textnormal{($0\leq n_{i}\leq q-1$, $n_{i}=0$ for $i\gg 0$)}.
\]Then the Carlitz factorial is defined by
\[
\Gamma_{n+1}:=\prod_{i=0}^{\infty} D_{i}^{n_{i}}\in A.
\]One of the major results in \cite{AT90} is to relate $\zeta_{A}(n)$ to the last coordinate of the logarithm of $\bC^{\otimes n}$. It is stated as follows.
\begin{theorem}{\rm{(Anderson-Thakur~\cite[Thm.~3.8.3]{AT90})}}\label{T:AT90}
For each positive integer $n$, one explicitly constructs an integral point $\bv_{n}\in \bC^{\otimes n}(A)$ so that there exists a vector $Y_{n}\in \CC_{\infty}^{n}$ of the form
\[ Y_{n}=\left(
           \begin{array}{c}
             * \\
             \vdots \\
             \Gamma_{n}\zeta_{A}(n) \\
           \end{array}
         \right)
    \] satisfying \[\exp_{n}(Y_{n})=\bv_{n}.\]
\end{theorem}

For an integer $m$, we define $m$-fold Frobenius twisting by
\[
     \begin{array}{rcl}
      \laurent{\CC_\infty}{t}  & \rightarrow & \laurent{\CC_\infty}{t},\\
       f:=\sum_{i}a_{i}t^{i} & \mapsto & f^{(m)}:=\sum_{i}{a_{i}}^{q^{m}}t^{i}. \\
     \end{array}   \]
We extend this to matrices with entries in $\laurent{\CC_\infty}{t}$ by twisting entry-wise.

We put $G_{0}(y):=1$ and define polynomials $G_{n}(y)\in \FF_{q}[t,y]$ for $n\in \NN$ by the product
\[
G_{n}(y)=\prod_{i=1}^{n}\left( t^{q^{n}}-y^{q^{i}} \right).
\] For $n=0,1,2,\ldots$, we define the sequence of Anderson-Thakur polynomials $H_{n}\in A[t]$ by the generating function identity
\[
\left( 1-\sum_{i=0}^{\infty} \frac{ G_{i}(\theta) }{ D_{i}|_{\theta=t}} x^{q^{i}}  \right)^{-1}=\sum_{n=0}^{\infty} \frac{H_{n}}{\Gamma_{n+1}|_{\theta=t}} x^{n}.
\]

We put
\[
\Omega(t):=(-\theta)^{\frac{-q}{q-1}} \prod_{i=1}^{\infty} \biggl(
1-\frac{t}{\theta^{q^{i}}} \biggr)\in \power{\CC_{\infty}}{t},\]
where $(-\theta)^{\frac{1}{q-1}}$ is a suitable choice of  $(q-1)$-st root of $-\theta$ so that $\frac{1}{\Omega(\theta)}=\tilde{\pi}$ (cf.~\cite{ABP04} and \cite{AT09}). The function $\Omega$ satisfies the difference equation $\Omega^{(-1)}=(t-\theta)\Omega $. One important identity established in \cite{AT90, AT09} is the following: for any positive integer $n$ and non-negative integer $i$, we have
\begin{equation}\label{E:OmegaHn}
\left(\Omega^{n} H_{n-1}\right)^{(i)}|_{t=\theta}= \frac{\Gamma_{n} S_{i}(n)}{\tilde{\pi}^{n}},
\end{equation}
where $S_{i}(n)$ is the partial sum
\[
 S_{i}(n):=\sum_{a\in A_{+,i}}\frac{1}{a^{n}}\in k.
\]
Here $A_{+,i}$ denotes by the set of all monic polynomials in $A$ with degree $i$. For any $(s_{1},\ldots,s_{r})\in \NN^{r}$ we define the following series
\begin{equation}\label{E:Lsi}
 \cL_{(s_{1},\ldots,s_{r})}(t):=\sum_{i_{1}> \cdots> i_{r}\geq 0}\left( \Omega^{s_{1}}H_{s_{1}-1} \right)^{(i_{1})}\cdots \left(\Omega^{s_{r}}H_{s_{r}-1} \right)^{(i_{r})}.
\end{equation}
Then by (\ref{E:OmegaHn}), specialization at $t=\theta$ of $\cL_{(s_{1},\ldots,s_{r})}$ gives
\[\cL_{(s_{1},\ldots,s_{r})}(\theta)=   \sum_{i_{1}> \cdots> i_{r}\geq 0}S_{i_{1}}(s_{1})\cdots S_{i_{r}}(s_{r})=\Gamma_{s_{1}}\cdots\Gamma_{s_{r}}\zeta_{A}(s_{1},\ldots,s_{r})/\tilde{\pi}^{s_{1}+\cdots+s_{r}}    .\]

We single out the following useful lemma, which is rooted in \cite[Lem.~5.3.5]{C14} (see also \cite[Prop.~2.3.3]{CPY14}).
\begin{lemma}\label{L:L theta qN}
For any $(s_{1},\ldots,s_{r})\in \NN^{r}$ and any nonnegative integer $N$, we have
\[  \cL_{(s_{1},\ldots,s_{r})}(\theta^{q^{N}})=\left(\Gamma_{s_{1}}\cdots\Gamma_{s_{r}}\zeta_{A}(s_{1},\ldots,s_{r})/\tilde{\pi}^{s_{1}+\cdots+s_{r}}    \right)^{q^{N}}    .\]
\end{lemma}

\subsubsection{Yu's theory} In \cite{Yu91}, Yu proved the transcendence of $\zeta_{A}(n)$ for each $n\in \NN$. The key ingredient in Yu's proof is to establish the following theorem.
\begin{theorem}{\rm{(Yu~\cite[Thm.~2.3]{Yu91})}}\label{T:Yu's thm} Let $n$ be a positive integer, and $Y=(y_{1},\ldots,y_{n})^{\tr}\in \CC_{\infty}^{n}$ be a nonzero vector  satisfying that $\exp_{n}(Y)\in \bC^{\otimes n}(\ok)$. Then $y_{n}$ is transcendental over $k$.
\end{theorem}

\begin{remark}\label{Rem:TorsionEven}
Combining Theorems~\ref{T:AT90} and \ref{T:Yu's thm}, one can show (see~\cite[Thm.~3.2]{Yu91}):
\begin{enumerate}
\item[$\bullet$] $\zeta_{A}(n)/\tilde{\pi}^{n}\in k$ if and only if $\bv_{n}$ is an $\FF_{q}[t]$-torsion point in $\bC^{\otimes n}(A)$.

\item[$\bullet$] $\bv_{n}$ is an $\FF_{q}[t]$-torsion point if and only if $n$ is divisible by $q-1$.
\end{enumerate}
\end{remark}

\subsection{Review of the CPY criterion for Eulerian MZV's}
In what follows, by a {\it{Frobenius}} module we mean a left $\ok[t,\sigma]$-module that is free of finite rank over $\ok[t]$, where $\ok[t,\sigma]:=\ok[t][\sigma]$ is the twisted polynomial ring generated by $\sigma$ over $\ok[t]$ subject to the relation $\sigma f =f^{(-1)}\sigma$ for $f\in \ok[t]$. Morphisms of Frobenius modules are defined to be left $\ok[t,\sigma]$-module homomorphisms and we denote by $\cF$ the category of Frobenius modules.

In what follows, an object $M$ in $\cF$ is said to be defined by a matrix $\Phi\in \Mat_{r}(\ok[t])$ if $M$ is free of rank $r$ over $\ok[t]$ and the $\sigma$-action on a given $\ok[t]$-basis of $M$ is represented by the matrix $\Phi$. We denote by ${\bf{1}}$ the {\it{trivial}} object in $\cF$, where the underlying space of ${\bf{1}}$ is $\ok[t]$ and on which $\sigma$ acts as $\sigma f=f^{(-1)}$  for $f\in {\bf{1}}$. We further denote by $C^{\otimes n}$ the $n$th tensor power of the Carlitz motive for $n\in \NN$. The underlying space of $C^{\otimes n}$ is $\ok[t]$, and on which $\sigma$ acts by $\sigma f:=(t-\theta)^{n} f^{(-1)}$ for $f\in C^{\otimes n}$.

 By an {\it{Anderson $t$-motive}} we mean an object $M'\in \cF$ that possesses the following properties.
\begin{enumerate}
\item[$\bullet$] $M'$ is also a free left $\ok[\sigma]$-module of finite rank.
\item[$\bullet$] $\sigma M'\subseteq (t-\theta)^{n}M'$ for all sufficiently large integers $n$.
\end{enumerate}
Here we follow the terminology of Anderson $t$-motives in \cite{P08} (cf.~\cite{A86, ABP04}).

For a fixed Anderson $t$-motive $M'$ of rank $d$ over $\ok[\sigma]$, we are interested in $\Ext_{\cF}^{1}\left( {\bf{1}},M' \right)$, the set of equivalence classes of Frobenius modules $M$ fitting into a short exact sequence of Frobenius modules
\[  0\rightarrow M' \hookrightarrow M \twoheadrightarrow {\bf{1}}\rightarrow 0 ,\] and we denote by $[M]$ the equivalence class of $M$ in $\Ext_{\cF}^{1}\left( {\bf{1}},M' \right)$. Since $\FF_{q}[t]$ is contained inside the center of $\ok[t,\sigma]$, left multiplication by any element of $\FF_{q}[t]$ on $M'$ is a morphism and hence $\Ext_{\cF}^{1}\left( {\bf{1}},M' \right)$ has a natural $\FF_{q}[t]$-module structure coming from Baer sum and pushout of morphisms of $M'$. The following is a review of the $\FF_{q}[t]$-module isomorphisms established by Anderson
\[ \Ext_{\cF}^{1}({\bf{1}},M')\cong M'/(\sigma-1)M'\cong E'(\ok), \]where
\begin{equation}\label{E:E',rho}
E'=(\GG_{a}^{d},\rho)
\end{equation}
 is the $t$-module over $\ok$ associated to $M'$ in the sense that the $\ok$-valued points of $E'$ is isomorphic to $M'/(\sigma-1)M'$ as $\FF_{q}$-vector spaces and the $\FF_{q}[t]$-module structure on $E'$ via $\rho$ is induced by the $\FF_{q}[t]$-action on $M'/(\sigma-1)M'$. For a detailed description of the isomorphisms above, see~\cite[\S~5.2]{CPY14}. We also refer the reader to \cite{S97, PR03, HP04, Ta10, BP, CP12, HJ16} for related discussions.

For example, let $n$ be a positive integer. Then the $n$th tensor power of Carlitz motive $C^{\otimes n}$ has a $\ok[\sigma]$-basis $\left\{ (t-\theta)^{n-1},\ldots,(t-\theta),1\right\}$ and every $f\in C^{\otimes n}/(\sigma-1)C^{\otimes n}$ has a unique representative polynomial (with degree $\leq n-1$) of the form $u_{1}(t-\theta)^{n-1}+\cdots+u_{n}\in \ok[t]$. Then the maps above can be characterized as
\begin{equation}\label{E:Isom C otimes n}
     \begin{array}{ccccc}
   \Ext_{\cF}^{1}({\bf{1}},C^{\otimes n}) &\cong &  C^{\otimes n}/(\sigma-1)C^{\otimes n}  & \cong & \bC^{\otimes n}(\ok) \\
    \left[M_{f}\right]  &\mapsto & f & \mapsto & (u_{1},\ldots,u_{n})^{\tr},\\
     \end{array}
  \end{equation}
  where $M_{f}\in \cF$ is defined by the matrix

\begin{equation}\label{E:Phi f}
\Phi_{f}:=\left(
    \begin{array}{cc}
      (t-\theta)^{n} & 0 \\
      f^{(-1)}(t-\theta)^{n} & 1 \\
    \end{array}
  \right)\in \Mat_{2}(\ok[t] ).
\end{equation}

\begin{remark}\label{Rem:special point}
For $n\in \NN$, we note that $H_{n-1}^{(-1)}(t-\theta)^{n}=H_{n-1}+(\sigma-1)H_{n-1}$ in $C^{\otimes n}$ and so
\[ H_{n-1}^{(-1)}(t-\theta)^{n}\equiv H_{n-1}\in C^{\otimes n}/(\sigma-1)C^{\otimes n}. \] The special point $\bv_{n}$ given in Theorem~\ref{T:AT90} is
defined to be the image of  $H_{n-1}^{(-1)}(t-\theta)^{n}$ under the isomorphism $C^{\otimes n}/(\sigma-1)C^{\otimes n}\cong \bC^{\otimes n}(\ok)$. For the  details, see \cite[p.~26]{CPY14}.
\end{remark}

 Let $Z$ be an MZV of weight $w$. Following \cite{T04}, we say that $Z$ is {\it{Eulerian}} if the ratio $Z/\tilde{\pi}^{w}$ is in $k$. In~\cite{CPY14}, an effective criterion for Eulerian MZV's is established and we describe it as follows. Let $r$ be a positive integer and fix an $r$-tuple $\fs=(s_{1},\ldots,s_{r})\in \NN^{r}$. We define the matrix $\Phi_{\fs} \in \Mat_{r+1}(\ok[t])$,
\begin{equation}\label{E:Phi s}
\Phi_{\fs} :=
               \begin{pmatrix}
                (t-\theta)^{s_{1}+\cdots+s_{r}}  & 0 & 0 &\cdots  & 0 \\
                H_{s_{1}-1}^{(-1)}(t-\theta)^{s_{1}+\cdots+s_{r}}  & (t-\theta)^{s_{2}+\cdots+s_{r}} & 0 & \cdots & 0 \\
                 0 &H_{s_{2}-1}^{(-1)} (t-\theta)^{s_{2}+\cdots+s_{r}} &  \ddots&  &\vdots  \\
                 \vdots &  & \ddots & (t-\theta)^{s_{r}} & 0 \\
                 0 & \cdots & 0 & H_{s_{r}-1}^{(-1)}(t-\theta)^{s_{r}} & 1 \\
               \end{pmatrix}.
\end{equation}
Define $ \Phi'_{\fs}\in \Mat_{r}(\ok[t])$ to be the square matrix of size $r$ cut off from the upper left square of $\Phi_{\fs}$:
\begin{equation}\label{E:Phi s'}
\Phi_{\fs}' :=
\begin{pmatrix}
                (t-\theta)^{s_{1}+\cdots+s_{r}}  &  &  &  \\
                H_{s_{1}-1}^{(-1)}(t-\theta)^{s_{1}+\cdots+s_{r}}  & (t-\theta)^{s_{2}+\cdots+s_{r}}   &  &  \\
                  & \ddots & \ddots &  \\
                  &  & H_{s_{r-1}-1}^{(-1)}(t-\theta)^{s_{r-1}+s_{r}} & (t-\theta)^{s_{r}}  \\
\end{pmatrix}.
\end{equation}

Define
\begin{equation}\label{E:Psi}
\Psi_{\fs}:=\left(
              \begin{array}{cccccc}
               \Omega^{s_{1}+\cdots+s_{r}}  &  &  &  &  &  \\
                \Omega^{s_{2}+\cdots+s_{r}}\cL_{s_{1}} & \Omega^{s_{2}+\cdots+s_{r}} &  &  &  &  \\
                \Omega^{s_{3}+\cdots+s_{r}}\cL_{(s_{1},s_{2})} &\Omega^{s_{3}+\cdots+s_{r}} \cL_{s_{2}} & \ddots &  &  &  \\
                \vdots & \vdots & \ddots & \Omega^{s_{r-1}+s_{r}} &  &  \\
                \Omega^{s_{r}}\cL_{(s_{1},\ldots,s_{r-1})} & \Omega^{s_{r}}\cL_{(s_{2},\ldots,s_{r-1})} &  & \Omega^{s_{r}} \cL_{s_{r-1}}& \Omega^{s_{r}} &  \\
                \cL_{(s_{1},\ldots,s_{r})} & \cL_{(s_{2},\ldots,s_{r})} & \cdots &\cL_{(s_{r-1},s_{r})}  & \cL_{s_{r}} &  1\\
              \end{array}
            \right)\in \GL_{r+1}(\TT)
\end{equation}
and let

\begin{equation}\label{E:Psi'}
\Psi_{\fs}':=\left(
              \begin{array}{ccccc}
               \Omega^{s_{1}+\cdots+s_{r}}  &  &  &  &    \\
                \Omega^{s_{2}+\cdots+s_{r}}\cL_{s_{1}} & \Omega^{s_{2}+\cdots+s_{r}} &  &  &    \\
                \Omega^{s_{3}+\cdots+s_{r}}\cL_{(s_{1},s_{2})} &\Omega^{s_{3}+\cdots+s_{r}} \cL_{s_{2}} & \ddots &  &    \\
                \vdots & \vdots & \ddots & \Omega^{s_{r-1}+s_{r}} &    \\
                \Omega^{s_{r}}\cL_{(s_{1},\ldots,s_{r-1})} & \Omega^{s_{r}}\cL_{(s_{2},\ldots,s_{r-1})} &  & \Omega^{s_{r}} \cL_{s_{r-1}}& \Omega^{s_{r}}   \\
              \end{array}
            \right)
   \in \GL_{r}(\TT)
\end{equation} be the square matrix cut off from the upper left square of $\Psi_{\fs}$. Then we have that
\begin{equation}\label{E:DiffPsiPsi'}
 \Psi_{\fs}'^{(-1)}=\Phi_{\fs}' \Psi_{\fs}'\hbox{ and }\Psi_{\fs}^{(-1)}=\Phi_{\fs}\Psi_{\fs}
\end{equation}
 (see \cite{AT09, C14, CPY14}).

We let $M_{\fs}$ (resp. $M_{\fs}'$) be the Frobenius module defined by the matrix $\Phi_{\fs}$ (resp. $\Phi_{\fs}'$). Then $M_{\fs}$ represents a class in $ \Ext_{\cF}^{1}\left({\bf{1}},M_{\fs}' \right)\cong M_{\fs}'/(\sigma-1)M_{\fs}'\cong E_{\fs}'(\ok)$ and we define $\bv_{\fs}\in E_{\fs}'(\ok)$ to be the image of $[M_{\fs}]$ under the composition of the isomorphisms above. Note that it is shown in \cite[Thm.~5.3.4]{CPY14} that actually $E_{\fs}'$ is defined over $A$ and $\bv_{\fs}$ is an integral point in $E_{\fs}'(A)$.

The criterion of Chang-Papanikolas-Yu for Eulerian MZV's is as follows.
\begin{theorem}{\rm{\cite[Thms.~5.3.5 and 6.1.1]{CPY14}}}\label{T:CPY-criterion}
For each $r$-tuple $\fs=(s_{1},\ldots,s_{r})\in \NN^{r}$, let $E_{\fs}'$ and $\bv_{\fs}$ be defined as above. Then we explicitly construct a polynomial $a_{\fs}\in \FF_{q}[t]$ so that $\zeta_{A}(\fs)$ is Eulerian if and only if $\bv_{\fs}$ is an $a_{\fs}$-torsion point in $E_{\fs}'(A)$.
\end{theorem}

\section{Step~I: A necessary condition}
\subsection{The formulation}
In this section, our main goal is to show the following necessary condition, which will be applied to compute the dimension of double zeta values.
\begin{theorem}\label{T:NecCond}
Let $n\geq 2$ be an integer and for $i=1,\ldots,m$, let $\fs_{i}=(s_{i1},s_{i2})\in \NN^{2}$ be chosen with $s_{i1}+s_{i2}=n$. Suppose that we have
\begin{equation}\label{E:c_i}
c_{0} \zeta_{A}(n)+ \sum_{i=1}^{m} c_{i}\zeta_{A}(\fs_{i})=0\hbox{ for some }c_{0},c_{1},\ldots,c_{m}\in k.
\end{equation}
If the coefficient $c_{j}$ is nonzero for some $1\leq j\leq m$, then we have that $(q-1)|s_{j2}$. In other words, all $k$-linear relations among  $\left\{\zeta_{A}(n), \zeta_{A}(\fs_{1}),\ldots,\zeta_{A}(\fs_{m})\right\}$ are those coming from the $k$-linear relations among $\left\{\zeta_{A}(n)\right\} \cup \left\{ \zeta_{A}(\fs_{j}); (q-1)|s_{j2} \right\}$.
\end{theorem}
\begin{proof}

Note that each double zeta value $\zeta_{A}(\fs_{i})$ is associated to the system of Frobenius difference equations
\begin{equation}\label{E:DiffPhi_i}
\Psi_{i}^{(-1)}=\Phi_{i} \Psi_{i},
\end{equation}
 where
\[\Phi_{i}:=\left(
               \begin{array}{ccc}
                (t-\theta)^{n}  & 0 & 0 \\
                H_{s_{i1}-1}^{(-1)}(t-\theta)^{n}  & (t-\theta)^{s_{i2}} & 0 \\
                 0 & H_{s_{i2}-1}^{(-1)}(t-\theta)^{s_{i2}} & 1 \\
               \end{array}
             \right)\]
  and
  \[ \Psi_{i}:=\left(
                 \begin{array}{ccc}
                   \Omega^{n}&0&0 \\
                   \Omega^{s_{i2}}\cL_{21}^{[i]}& \Omega^{s_{i2}}&0 \\
                   \cL_{31}^{[i]}& \cL_{32}^{[i]} &1  \\
                 \end{array}
               \right).
     \]Here for each $1\leq i\leq m$,
     \[ \cL_{21}^{[i]}:=\cL_{s_{i1}} =\sum_{\ell=0}^{\infty} \left( \Omega^{s_{i1}} H_{s_{i1}-1} \right)^{(\ell)}\in \TT     \]
and \[\cL_{31}^{[i]}:=\cL_{(s_{i1},s_{i2})}:=\sum_{\ell_{1}>\ell_{2}\geq 0}\left( \Omega^{s_{i2}}H_{s_{i2}-1} \right)^{(\ell_{2})}\left(\Omega^{s_{i1}}H_{s_{i1}-1} \right)^{(\ell_{1})} \in \TT, \]which
satisfy for any integer $N\in \ZZ_{\geq 0}$,
\begin{equation}\label{E:formulaL21L31}
 \cL_{21}^{[i]}(\theta^{q^{N}})=\left(\zeta_{A}(s_{i1})/\tilde{\pi}^{n}\right)^{q^{N}} {\hbox{ and }} \cL_{31}^{[i]}(\theta^{q^{N}})=\left(\zeta_{A}(s_{i1},s_{i2})/\tilde{\pi}^{n}\right)^{q^{N}}\hbox{ (see~Lemma~\ref{L:L theta qN})} .
\end{equation}

Moreover, $\zeta_{A}(n)$ is associated to the system of Frobenius difference equations

\begin{equation}\label{E:DiffZeta(n)}
 \left(
     \begin{array}{c}
       \Omega^{n} \\
       \cL_{n} \\
     \end{array}
   \right)^{(-1)}=\left(
                    \begin{array}{cc}
                      (t-\theta)^{n} & 0 \\
                      H_{n-1}^{(-1)}(t-\theta)^{n} & 1 \\
                    \end{array}
                  \right)\left(
                           \begin{array}{c}
                             \Omega^{n} \\
                             \cL_{n} \\
                           \end{array}
                         \right),
   \end{equation} where
   \[ \cL_{n}:=\sum_{i=0}^{\infty}\left(\Omega^{n}H_{n-1} \right)^{(i)}\in \TT  \]
has the property that for $N\in \ZZ_{\geq 0}$,
\[  \cL_{n}(\theta^{q^{N}})=\left(\Gamma_{n}\zeta_{A}(n)/\tilde{\pi}^{n} \right)^{q^{N}}  {\hbox{(see~Lemma~\ref{L:L theta qN})}} .\]

 To prove this theorem, without loss of generality we assume that each coefficient $c_{i}\neq 0$ for $i=1,\ldots,m$. Multiplying the equation (\ref{E:c_i}) by a suitable element in $\FF_{q}[\theta]$, without loss of generality we may assume that $a_{i}:=\frac{c_{i}}{\Gamma_{s_{i1}}\Gamma_{s_{i2}}}|_{\theta=t}$ and $a_{0}:=\frac{c_{0}}{\Gamma_{n}}|_{\theta=t}$ are in $\FF_{q}[t]$ for $i=1,\ldots,m$. Note that $a_{1}\neq 0,\ldots,a_{m}\neq 0$. So we have that
\[ a_{0}(\theta)\Gamma_{n}\zeta_{A}(n)+ \sum_{i=1}^{m} a_{i}(\theta)\Gamma_{s_{i1}}\Gamma_{s_{i2}} \zeta(\fs_{i})=0.  \]
Define
\[ \Phi:=\left(
           \begin{array}{ccccc}
             (t-\theta)^{n} &  &  &  &  \\
             H_{s_{11}-1}^{(-1)}(t-\theta)^{n} & (t-\theta)^{s_{12}} &  &  &  \\
             \vdots &  & \ddots &  &  \\
             H_{s_{m1}-1}^{(-1)}(t-\theta)^{n} &  &  & (t-\theta)^{s_{m2}} &  \\
             a_{0}H_{n-1}^{(-1)}(t-\theta)^{n} & a_{1}H_{s_{12}-1}^{(-1)}(t-\theta)^{s_{12}} & \cdots & a_{m}H_{s_{m2}-1}^{(-1)}(t-\theta)^{s_{m2}} & 1 \\
           \end{array}
         \right)\in \Mat_{m+2}(\ok[t])
  \] and
  \[  \psi:= \left(
               \begin{array}{c}
                 \Omega^{n} \\
                 \Omega^{s_{12}}\cL_{21}^{[1]} \\
                 \vdots \\
                 \Omega^{s_{m2}}\cL_{21}^{[m]} \\
                 a_{0}\cL_{n}+\sum_{i=1}^{m}a_{i}\cL_{31}^{[i]} \\
               \end{array}
             \right)\in \Mat_{(m+2)\times 1}(\TT)
   .\] Using (\ref{E:DiffPhi_i}) and (\ref{E:DiffZeta(n)}) one has $\psi^{(-1)}=\Phi \psi$.

By hypothesis we have \[\left( a_{0}\cL_{n}+\sum_{i=1}^{m}a_{i}\cL_{31}^{[i]}\right) (\theta)= \frac{c_{0}\zeta_{A}(n)+\sum_{i=1}^{m}c_{i}\zeta_{A}(\fs_{i}) }{ \tilde{\pi}^{n} } =0.\]It follows from the ABP-criterion~\cite[Thm.~3.1.1]{ABP04} that there exists ${\bf{f}}=(f_{0},f_{1},\ldots,f_{m},f_{m+1})\in \Mat_{1 \times (m+2)}(\ok[t])$ so that
\[ {\bf{f}}\psi=0\hbox{ and }{\bf{f}}(\theta)=(0,\ldots,0,1)     .\]
Put $\tilde{\bf{f}}:=\frac{1}{f_{m+1}}{\bf{f}}$ and note that $\tilde{\bf{f}}\psi=0$. We take the $(-1)$-fold Frobenius twist of the equation $\tilde{\bf{f}}\psi=0$ and then subtract it from itself, and then we have that
\[ (\tilde{\bf{f}}-\tilde{\bf{f}}^{(-1)}\Phi)\psi=0.   \]
Note that the last coordinate of the vector $\tilde{\bf{f}}-\tilde{\bf{f}}^{(-1)}\Phi$ is zero. Define
\[ (R_{0},R_{1},\ldots,R_{m},0):=\tilde{\bf{f}}-\tilde{\bf{f}}^{(-1)}\Phi \]
and note that
\begin{equation}\label{E:Ri}
  \begin{array}{rl}
    R_{0}:= & \tilde{f_{0}}-\tilde{f_{0}}^{(-1)}(t-\theta)^{n}-\sum_{i=1}^{m}\tilde{f_{i}}^{(-1)} H_{s_{i1}-1}^{(-1)}(t-\theta)^{n}-a_{0}H_{n-1}^{(-1)}(t-\theta)^{n}\\
    R_{1}:= & \tilde{f_{1}}-\tilde{f_{1}}^{(-1)}(t-\theta)^{s_{12}}-a_{1}H_{s_{12}}^{(-1)}(t-\theta)^{s_{12}} \\
     \vdots&  \\
    R_{m}:= & \tilde{f_{m}}-\tilde{f_{m}}^{(-1)}(t-\theta)^{s_{m2}}-a_{m}H_{s_{m2}-1}^{(-1)}(t-\theta)^{s_{m2}}, \\
  \end{array}
\end{equation} where $(\tilde{f_{1}},\cdots,\tilde{f_{m}},0):=\tilde{\bf{f}}$.
We claim that $R_{0}=R_{1}=\cdots=R_{m}=0$.

Assume this claim first. Put
\[\gamma:= \left(
    \begin{array}{ccccc}
      1 &  &  &  &  \\
       & 1 &  &  &  \\
       &  & \ddots &  &  \\
       &  &  & 1 &  \\
     \tilde{f_{0}}  & \tilde{f_{1}} &\ldots  & \tilde{f_{m}} & 1 \\
    \end{array}
  \right)
    \] and note that the claim above implies the following difference equations
    \[\gamma^{(-1)} \Phi=\left(
                    \begin{array}{cc}
                      \Phi' &  \\
                       & 1 \\
                    \end{array}
                  \right)\gamma
       ,\] where $\Phi'$ is the square matrix of size $m+1$ cut off from the upper left square of $\Phi$.

       By \cite[Prop.~2.2.1]{CPY14} the rational functions $\tilde{f_{0}},\ldots,\tilde{f_{m}}$ have a (nonzero) common denominator $b\in \FF_{q}[t]$ so that
$b\tilde{f_{i}}\in \ok[t]$ for $i=0,1,\ldots,m$. Since $b^{(-1)}=b$, multiplication by $b$ on the both sides of (\ref{E:Ri}) shows that if we put
\[ \delta:= \left(
    \begin{array}{ccccc}
      1 &  &  &  &  \\
       & 1 &  &  &  \\
       &  & \ddots &  &  \\
       &  &  & 1 &  \\
    b \tilde{f_{0}}  & b\tilde{f_{1}} &\ldots  & b\tilde{f_{m}} & 1 \\
    \end{array}
  \right) \] and $\nu:=(ba_{0}H_{n-1}^{(-1)}(t-\theta)^{n},ba_{1}H_{s_{12}-1}^{(-1)}(t-\theta)^{s_{12}},\ldots,ba_{m}H_{s_{m2}-1}^{(-1)}(t-\theta)^{s_{m2}})\in \Mat_{(m+1)\times 1}(\ok[t])$, then we have

  \begin{equation}\label{E:delta nu}
   \delta^{(-1)}\left(
                         \begin{array}{cc}
                           \Phi' &  \\
                           \nu & 1 \\
                         \end{array}
                       \right) =\left(
                         \begin{array}{cc}
                           \Phi' &  \\
                            & 1 \\
                         \end{array}
                       \right) \delta.
       \end{equation}

       Let $M'$ (resp. $M$) be the Frobenius module defined by $\Phi'$ (resp. $\Phi$) and note that $[M]\in \Ext_{\cF}^{1}\left({\bf{1}},M' \right)$. One checks directly that $M'$ is an Anderson $t$-motive (cf. the proofs of \cite[Prop.~6.1.3]{P08} and \cite[Lem.~A.1]{CY07}). So the action of $b$ on $M$ denoted by $b*M\in \cF$, is  defined by the matrix
       \[ b*\Phi:=\left(
                                                                                                    \begin{array}{cc}
                                                                                                      \Phi' &  \\
                                                                                                      \nu & 1 \\
                                                                                                    \end{array}
                                                                                                  \right)
        \] (see~\cite[\S\S~2.4]{CPY14}). So the system of difference equations (\ref{E:delta nu}) implies that $b*M$ represents the trivial class in $\Ext_{\cF}^{1}\left({\bf{1}},M' \right)$. Furthermore, if we let $\left\{x_{0},x_{1},\ldots,x_{m}\right\}$ be a $\ok[t]$-basis of $M'$ on which the action of $\sigma$ is represented by $\Phi'$, then we have the isomorphism of $\FF_{q}[t]$-modules (see~\cite[Thm.~5.2.1]{CPY14}):
       \[
            \begin{array}{rcl}
             \Ext_{\cF}^{1}\left({\bf{1}},M' \right)  & \cong & M'/(\sigma-1)M' \\
           \left[ b*M \right]  & \mapsto & \widetilde{b*M}:= ba_{0}H_{n-1}^{(-1)}(t-\theta)^{n}x_{0}+ \sum_{i=1}^{m}b a_{i}H_{s_{i2}-1}^{(-1)}(t-\theta)^{s_{i2}}x_{i}  +(\sigma-1)M' .\\
            \end{array}
         \]

         Note that $M'$ fits into the short exact sequence of Frobenius modules
         \[ 0\rightarrow C^{\otimes n}\rightarrow M' \twoheadrightarrow \oplus_{i=1}^{m} C^{\otimes s_{i2}}\rightarrow 0  ,\]
         where the projection map $\pi: M' \twoheadrightarrow \oplus_{i=1}^{m} C^{\otimes s_{i2}}$ is given by $\pi:=\left( \sum_{i=0}^{m}g_{i}x_{i}\mapsto (g_{1},\ldots,g_{m}) \right)$. However, since the $\FF_{q}[t]$-linear map $(\sigma-1):\oplus_{i=1}^{m} C^{\otimes s_{i2}}\rightarrow \oplus_{i=1}^{m} C^{\otimes s_{i2}} $ is injective, the snake lemma shows that we have the following short exact sequence of $\FF_{q}[t]$-modules:
         \[
              \begin{array}{ccccccccc}
                0 &\rightarrow  & C^{\otimes n}/(\sigma-1)C^{\otimes n} & \rightarrow & M'/(\sigma-1)M' & \twoheadrightarrow & \oplus_{i=1}^{m}\left( C^{\otimes s_{i2}}/(\sigma-1)C^{\otimes s_{i2}}\right) & \rightarrow & 0 .\\
                 &  &  &  & \widetilde{b* M} &\mapsto  & \left( ba_{i}H_{s_{i2}-1}^{(-1)}(t-\theta)^{s_{i2}} \right)_{i} &  &  \\
              \end{array}
           \]

           For any $s\in \NN$ we  recall the identification $C^{\otimes s}/(\sigma-1)C^{\otimes s}\cong \bC^{\otimes s}(\ok)$, which is an $\FF_{q}[t]$-module isomorphism and under which $H_{s-1}^{(-1)}(t-\theta)^{s}$ is mapped to the special point $\bv_{s}\in \bC^{\otimes s}(A)$ (by Remark~\ref{Rem:special point}), itself associated to $\zeta_{A}(s)$. It follows that $(ba_{i}H_{s_{i2}-1}^{(-1)}(t-\theta)^{s_{i2}} )_{i}$ is mapped to $([ba_{i}]_{s_{i2}} \bv_{s_{i2}}  )_{i}\in \oplus_{i=1}^{m}\bC^{\otimes s_{i2}}(A)$. Since $b*M$ represents the trivial class in $\Ext_{\cF}^{1}\left({\bf{1}},M' \right)$,  $\widetilde{b*M}$ is also trivial in $M'/(\sigma-1)M'$ and hence $[ba_{i}]\bv_{s_{i2}}={\bf{0}}\in \bC^{\otimes s_{i2}}(A)$ for each $1\leq i\leq m$. That is, each $\bv_{s_{i2}}$ is $ba_{i}$-torsion as $ba_{i}$ is nonzero. It follows by Remark~\ref{Rem:TorsionEven} that $s_{i2}$ must be divisible by $q-1$ for each $1\leq i \leq m$.

To finish the proof, it suffices to prove the claim above. Since $s_{i1}+s_{i2}=n$ for each $1\leq i\leq m$, without loss of generality we may assume that $s_{12}>s_{22}>\cdots>s_{m2}$. From the equation $(\tilde{\bf{f}}-\tilde{\bf{f}}^{(-1)}\Phi)\psi=0$ we have that
\begin{equation}\label{E:RL}
R_{0}\Omega^{n}+R_{1}\Omega^{s_{12}}\cL_{21}^{[1]}+\cdots+R_{m}\Omega^{s_{m2}}\cL_{21}^{[m]}=0.
\end{equation}
Dividing the equation above by $\Omega^{s_{m2}}$ one has
\begin{equation}\label{E:RL2}
R_{0}\Omega^{s_{m1}}+R_{1}\Omega^{s_{12}-s_{m2}}\cL_{21}^{[1]}+\cdots+R_{m}\cL_{21}^{[m]}=0.
\end{equation}
Note that $\Omega$ has simple zero at each $t=\theta^{q^{N}}$ for each $N\in \NN$. Since each $R_{i}$ is a rational function having only finitely many poles, we can pick a sufficiently large integer $N$ so that  $R_{i}$ is regular at $t=\theta^{q^{N}}$ for $i=1,\ldots,m$. Specializing both sides of the equation (\ref{E:RL2}) at $t=\theta^{q^{N}}$ and using the formula (\ref{E:formulaL21L31}) show that
\[ R_{m}(\theta^{q^{N}})\left(\zeta_{A}(s_{m2})/\tilde{\pi}^{s_{m2}} \right)^{q^N}=0 .\]
Since each MZV is non-vanishing by \cite{T09a} and the equality above is valid for $N\gg 0$, $R_m$ has to be zero as it has only finitely many zeros.

We turn back to the equation (\ref{E:RL}) and repeat the arguments above. We then eventually have $R_{m}=\cdots=R_{2}=0$ and so obtain
\[ R_{0}\Omega^{n}+R_{1}\Omega^{s_{12}}\cL_{21}^{[1]}=0   .\]
Again, by dividing $\Omega^{s_{12}}$, the arguments above show that $R_{1}=0$ and hence $R_{0}=0$.

\end{proof}

\begin{corollary}\label{C:NecCondDepth2}
Let $n\geq 2$ be an integer and  for $i=1,\ldots,m$, let $\fs_{i}=(s_{i1},s_{i2})\in \NN^{2}$ be chosen with $s_{i1}+s_{i2}=n$. Suppose that we have
\[
c_{0} \tilde{\pi}^{n}+ \sum_{i=1}^{m} c_{i}\zeta_{A}(\fs_{i})=0\hbox{ for some }c_{0},c_{1},\ldots,c_{m}\in k.
\]
If the coefficient $c_{j}$ is nonzero for some $1\leq j\leq m$, then we have that $(q-1)|s_{j2}$. In other words, all $k$-linear relations among  $\left\{\tilde{\pi}^{n}, \zeta_{A}(\fs_{1}),\ldots,\zeta_{A}(\fs_{m})\right\}$ are those coming from the $k$-linear relations among $\left\{\tilde{\pi}^{n}\right\} \cup \left\{ \zeta_{A}(\fs_{j}); (q-1)|s_{j2} \right\}$.
\end{corollary}
\begin{proof}
If $n$ is $A$-even, then $\zeta_{A}(n)/\tilde{\pi}^{n}\in k$ by Carlitz~\cite{Ca35}. So the result follows from Theorem~\ref{T:NecCond}. If $n$ is {\it{$A$-odd}} (ie., $n$ is not divisible by $q-1$), then $\tilde{\pi}^{n}\notin k_{\infty}$ and hence $\tilde{\pi}^{n}$ is $k$-linearly independent from all MZV's as each MZV is in $k_{\infty}$. Therefore, we can put $c_{0}=0$ in Theorem~\ref{T:NecCond} and the desired result follows.
\end{proof}

\begin{remark}
The result above shows that the set $\left\{\tilde{\pi}^{n}\right\} \cup \left\{ \zeta_{A}(\fs_{j}); (q-1)\nmid s_{j2} \right\}$ is linearly independent over $k$. It verifies the parity conjecture of Thakur~\cite{T09b} in this depth two setting.
\end{remark}

\begin{remark}\label{Rem:comparison}
For an even integer  $n>2$,  Gangl, Kaneko and Zagier~\cite[Thms.~1 and 2]{GKZ06} showed that
\[ \mathcal{O}:=(2\pi \sqrt{-1})^{n}\cup \left\{\zeta(3,n-3),\zeta(5,n-5),\ldots,\zeta(n-3,3)   \right\}  \]
is a set of generators for the vector space $\mathrm{DZ}_{n}$. Our point of view is that the set $\mathcal{O}$ is generated by $(2\pi\sqrt{-1})^{n}$ and weight $n$ double zeta values $\zeta(odd,odd)$, but excluding $\zeta(n-1,1)$ ($\zeta$ is not defined at $1$ and so the odd $1$ is special). Note that $\frac{n}{2}-1$ is the cardinality of $\mathcal{O}$ and hence Conjecture~\ref{Conj1} in the case of even $n> 2$ is equivalent to that there are $\dim_{\CC}S_{n}(SL_{2}(\ZZ))$ independent $\QQ$-linear relations among the set $\mathcal{O}$.

When the given weight $n\geq 2$ is $A$-even,  we consider the set
\[ \mathcal{O}_{A}:=\left\{ \tilde{\pi}^{n}\right\}\cup\left\{\zeta_{A}(s_{1},s_{2})|s_{1}+s_{2}=n\hbox{ and }(q-1)\nmid s_{2}  \right\},   \]
which is the analogue of the set $\mathcal{O}$ since each $\zeta_{A}(s_{1},s_{2})\in \mathcal{O}_{A}$ has the property that $s_{1}$ and $s_{2}$ are both $A$-odd. So it is not analogous to Conjecture~\ref{Conj1} as $\mathcal{O}_{A}$ is $k$-linearly independent by Corollary~\ref{C:NecCondDepth2}.
\end{remark}

\section{Step~II: Logarithmic Interpretation}
\subsection{The formulation}
In this section, our goal is to establish the following result.
\begin{theorem}\label{T1:LogMZV}
Let $r\geq 2$ be an integer and let $\fs:=(s_{1},\ldots,s_{r})\in \NN^{r}$ with $n:=\sum_{i=1}^{r}s_{i}$. Put $\fs':=(s_{2},\ldots,s_{r})$ and suppose that $\zeta_{A}(\fs')$ is Eulerian, ie., $\zeta_{A}(s_{2},\ldots,s_{r})/\tilde{\pi}^{s_{2}+\cdots+s_{r}}\in k$. Let $\alpha_{\fs}:=a_{\fs'}\in \FF_{q}[t]$ be given in Theorem~\ref{T:CPY-criterion}.
Then we explicitly construct an integral point $\Xi_{\fs}\in \bC^{\otimes n}(A)$
so that there exists a vector $Y_{\fs}\in \CC_{\infty}^{ n}$ of the form
\[Y_{\fs}=\left(
            \begin{array}{c}
              * \\
              \vdots \\
             \alpha_{\fs}(\theta)\Gamma_{s_{1}}\cdots\Gamma_{s_{r}}\zeta_{A}(s_{1},\ldots,s_{r})  \\
            \end{array}
          \right)
   \]
satisfying  \[\exp_{ n}(Y_{\fs})=\Xi_{\fs}.\]
\end{theorem}

 We divide the proof into the following steps.

\subsection{Difference equations arising from algebraic points of $\bC^{\otimes n}$}
For each positive integer $n$, we denote by $\log_{n}$ the logarithm of $\bC^{\otimes n}$. We first recall the convergence domain $\mathbb{D}_{n}$ of $\log_{n}$ (see \cite[Prop.~2.4.3]{AT90}):

\[ \mathbb{D}_{n}:= \left\{ \bz:=\left(z_1,\ldots,z_n \right)^{\tr}\in \bC^{\otimes n}(\CC_{\infty});\hbox{ } |z_{i}|_{\infty}< |\theta|_{\infty}^{i-n+\frac{nq}{q-1}}\hbox{ for }i=1,\ldots,n.  \right\} \]
For a fixed nonzero point $\bu=(u_{1},\ldots,u_{n})^{\tr}\in \bC^{\otimes n}(\ok)\cap \mathbb{D}_{n}$, we define the following polynomial associated to $\bu$:
\begin{equation}\label{E:fu}
 f:=f_{\bu}:=u_{1}(t-\theta)^{n-1}+\cdots+u_{n-1}(t-\theta)+u_{n} \in \ok[t] .
\end{equation}
 We let $M_{f}$ be the Frobenius module defined by the matrix $\Phi_{f}$ in \eqref{E:Phi f} and note that $[M_{f}]\in \Ext_{\cF}^{1}\left( {\bf{1}},C^{\otimes n} \right)$.

Now we define the series
\begin{equation}\label{E:Lf}
 \cL_{f}:=f+\sum_{i=1}^{\infty}\frac{f^{(i)} }{ (t-\theta^{q})^{n}\ldots(t-\theta^{q^{i}})^{n} } \in \power{\ok}{t} ,
 \end{equation}
and note that this kind of series was introduced by Papanikolas~\cite{P08} and later on studied in \cite{CY07, C14, M14, CPY14}.

\begin{proposition}
Let notation and assumptions be as above. Then there exists $\Psi_{f}\in \GL_{2}(\TT)$ so that $\Psi_{f}^{(-1)}=\Phi_{f} \Psi_{f}$. So $M_{f}$ is a $t$-motive in the sense of \cite{P08}.
\end{proposition}
\begin{proof}
Note that by the definition of $\cL_{f}$ we have
\[\cL_{f}^{(-1)}=f^{(-1)}+\sum_{i=0}^{\infty}\frac{f^{(i)}}{ (t-\theta)^{n}\cdots(t-\theta^{q^{i}})^{n}}=f^{(-1)}+\frac{\cL_{f}}{(t-\theta)^{n}},\] and hence
\begin{equation}\label{E:OmegaLf}
\left( \Omega^{n}\cL_{f} \right)^{(-1)}=f^{(-1)}(t-\theta)^{n} \Omega^{n} +\Omega^{n} \cL_{f}.
\end{equation}

It follows that if we put
\begin{equation}\label{Def:Psif}
 \Psi_{f}:=\left(
               \begin{array}{cc}
                 \Omega^{n} & 0 \\
                 \Omega^{n}\cL_{f} & 1 \\
               \end{array}
             \right),
 \end{equation}  then we have
 \[ \Psi_{f}^{(-1)}=\Phi_{f}\Psi_{f}.  \]

The hypotheses of $\bu$ imply that $\Omega^{n}\cL_{f}$ satisfies the condition of \cite[Lemma~5.3.1]{C14}, whence $\Omega^{n}\cL_{f}$ is an entire function and therefore $M_{f}$ is a $t$-motive ( cf.~\cite[Prop.~6.1.3]{P08} and \cite[\S\S~3.2]{CY07}).
\end{proof}

\subsection{Some Lemmas}
\begin{lemma}\label{L:LogPeriod}
Let $\bu$ be a nonzero point in $\bC^{\otimes n}(\ok)\cap\mathbb{D}_{n}$, and let $y$ be the last coordinate of $\log_{n}(\bu)$. Let $f\in \ok[t]$ be the polynomial associated to $\bu$ given in (\ref{E:fu}) and $\Psi_{f}$ be defined in \eqref{Def:Psif}.  Then we have
\[\Psi_{f}(\theta)=\left(
                     \begin{array}{cc}
                       1/\tilde{\pi}^{n} & 0 \\
                       y/\tilde{\pi}^{n} & 1 \\
                     \end{array}
                   \right)
   .\]
\end{lemma}

To prove Lemma~\ref{L:LogPeriod}, we need the following formula due to Papanikolas.

\begin{proposition}{\rm{(Papanikolas~\cite{P14})}}\label{T:CoeffPi} We write $\log_{n}=\sum_{i=0}^{\infty}P_{i}\tau^{i}$, where $P_{0}=I_{n}$ and  $P_{i}\in \Mat_{n}(k)$.  For each positive integer $i$, the bottom row vector of $P_{i}$ is given by
\[ \left( \frac{(-1)^{n-1}(\theta^{q^{i}}-\theta)^{n-1} }{ L_{i}^{n}},\ldots,  \frac{(-1)^{n-\ell}(\theta^{q^{i-\ell}}-\theta)^{n-\ell} }{ L_{i}^{n}} ,\ldots,\frac{1}{L_{i}^{n}} \right) .\]
\end{proposition}

{\it{Proof of Lemma~\ref{L:LogPeriod}}}. As we know that $\Omega^{n}(\theta)=1/\tilde{\pi}^{n}$, by (\ref{Def:Psif}) it suffices to show that
\[ \cL_{f}(\theta)=y.   \]
We interpret $\cL_{f}$ as
\[\cL_{f}= \left( u_{1}(t-\theta)^{n-1}+\cdots+u_{n-1}(t-\theta) +u_{n}\right)   +\sum_{i=1}^{\infty}\frac{(t-\theta^{q^{i}})^{n-1}u_{1}^{q^{i}}+\cdots+(t-\theta^{q^{i}})u_{n-1}^{q^{i}}+u_{n}^{q^{i}} }{ (t-\theta^{q})^{n}\ldots(t-\theta^{q^{i}})^{n} } . \]
By specializing at $t=\theta$ we have
\[ \cL_{f}(\theta)=u_{n}+\sum_{i=1}^{\infty}\frac{(-1)^{n-1} (\theta^{q^{i}}-\theta)^{n-1}u_{1}^{q^{i}}+\cdots+ (-1)(\theta^{q^{i}}-\theta)u_{n-1}^{q^{i}}+u_{n}^{q^{i}} }{ L_{i}^{n} }  . \]
Hence by Proposition~\ref{T:CoeffPi} we obtain that $y=\cL_{f}(\theta)$.

\begin{lemma}\label{L:DivisionPoint}
Let $n$ be a positive integer and  let $\Xi\in \bC^{\otimes n}(\ok)$ be an algebraic point. Then there exists a positive integer $m$ and an algebraic point $\bu\in \bC^{\otimes n}(\ok)$ (depending on $m$) satisfying
\begin{enumerate}
\item[$\bullet$] $[t^{m}]_{n}( \bu)=\Xi$ (ie., $\bu$ is a $t^{m}$-division point of $\Xi$),
\item[$\bullet$] $\log_{n}(\bu)$ converges.
\end{enumerate}
\end{lemma}
\begin{proof}
We recall that $\exp_{n}$ is an entire function on $\CC_{\infty}^{n}$ and is of the form $\exp_{n}=I_{n}+\sum_{i=1}^{\infty}Q_{i}\tau^{i}$. By the inverse function theorem, there exists open subsets $U,V\subset \CC_{\infty}^{n}$ so that $\exp_{n}:U\rightarrow V$ is one to one and its inverse is also continuous. Note that as formal power series $\log_{n}$ is inverse to $\exp_{n}$ and so $\log_{n}$ is defined on $V$. Since $\exp_{n}$ is surjective, there exists a vector $X\in \CC_{\infty}^{n}$ for which $\exp_{n}\left( X\right)=\Xi$.

For each positive integer $m$, $\partial [t^{m}]_{n}$ is an upper triangular matrix with $\theta^{m}$ down to the diagonals, and the other nonzero entries off the diagonals have degrees in $\theta$ strictly less than $m$. It follows that if we define $\parallel \partial[t^{m}]_{n}^{-1} \parallel$ to be the maximum of the absolute values of entries of $\partial[t^{m}]_{n}^{-1} $, then we have \[\parallel  \partial[t^{m}]_{n}^{-1} \parallel< |\theta|_{\infty}^{-m} .\]
 Hence, we can pick a sufficiently large integer $m$ for which
 \[   \partial[t^{m}]_{n}^{-1} \left( X\right)\in U  .\]
Put $\bu:=\exp_{n}\left( \partial[t^{m}]_{n}^{-1} \left( X\right)   \right)\in V$, at which $\log_{n}$ converges. By the functional equation of $\exp_{n}$, we observe that
\[[t^{m}]_{n}\left( \bu\right)=[t^{m}]_{n}\left(  \exp_{n}\left( \partial[t^{m}]_{n}^{-1} \left( X\right)   \right) \right)=\exp_{n}\left(X \right)=\Xi   .\]
\end{proof}

\subsection{Proof of Theorem~\ref{T1:LogMZV}}  Put $n:=s_{1}+\cdots+s_{r}$ and $\fs':=(s_{2},\ldots,s_{r})\in \NN^{r-1}$. We let $(E_{\fs}',\bv_{\fs},a_{\fs})$ (resp. $(E_{\fs'}',\bv_{\fs'},a_{\fs'})$) be given in Theorem~\ref{T:CPY-criterion} corresponding to $\fs$ (resp. $\fs'$). Note that in this setting we have the following exact sequence of $t$-modules defined over $A$:
\[
\xymatrix{
0\ar[r] & \bC^{\otimes n}(\bar{k}) \ar[r]  & E_{\fs}'(\bar{k}) \ar@{->>}[r]^{\pi}  & E_{\fs'}'(\bar{k}) \ar[r]  & 0,
}
\]
and note that the projection $\pi$ maps $\bv_{\fs}$ to $\bv_{\fs'}$ (see the proof \cite[Thm.~6.1.1]{CPY14}).

To simplify the notation, we drop the subscript and put $\alpha:=\alpha_{\fs}:=a_{\fs'}$. Since by hypothesis $\zeta_{A}(\fs')$ is Eulerian, Theorem~\ref{T:CPY-criterion} shows that $\bv_{\fs'}$ is an $\alpha$-torsion point in $E_{\fs'}'$. It follows that $\Xi_{\fs}:=\rho_{\alpha} (\bv_{\fs})\in \Ker\pi$, and so we  identify $\Xi_{\fs}$ as a point in $\bC^{\otimes n}(A)$. Recall that $\rho$ is the map defining the $\FF_{q}[t]$-module structure on $E_{\fs}'$ (see (\ref{E:E',rho})).

We write $\Xi_{\fs}=(v_{1},\ldots,v_{n})^{\tr}\in \bC^{\otimes n}(A)$ and let $g$ be the polynomial in $A[t]$ associated to $\Xi_{\fs}$ in (\ref{E:fu}). By Lemma~\ref{L:DivisionPoint} there exists a positive integer $m$ and $\bu_{\fs}\in \bC^{\otimes n}(\ok)$ (depending on $m$) so that
 $[t^{m}]_{n}( \bu_{\fs})=\Xi_{\fs}$ and $\log_{n}(\bu_{\fs})$ converges. We write $\bu_{\fs}=(u_{1},\ldots,u_{n})^{\tr}\in \bC^{\otimes n}(\ok)$ and let $f$ be the polynomial in $\ok[t]$ associated to $\bu_{\fs}$ in (\ref{E:fu}).

We recall the isomorphism $\Delta_{n}:\Ext_{\cF}^{1}\left( {\bf{1}},C^{\otimes n} \right)\cong \bC^{\otimes n}(\ok)$ and note that $\Delta_{n}$ maps $[M_{f}]$ (resp. $[M_{g}]$) to $\bu_{\fs}$ (resp. $\Xi_{\fs}$) by \eqref{E:Isom C otimes n}, where $M_{f}$ (resp. $M_{g}$) is Frobenius module defined by the matrix
\[\Phi_{f}:= \left(
     \begin{array}{cc}
       (t-\theta)^{n} & 0 \\
       f^{(-1)}(t-\theta)^{n} & 1 \\
     \end{array}
   \right)
  {\rm{(}}\hbox{ resp. } \Phi_{g}:= \left(
     \begin{array}{cc}
       (t-\theta)^{n} & 0 \\
       g^{(-1)}(t-\theta)^{n} & 1 \\
     \end{array}
   \right) {\rm{)}}.\]

Note that
\[ \Delta_{n}\left( [t^{m}*M_{f}] \right)=[t^{m}]_{n}\left(\Delta_{n}( [M_{f}])\right)=[t^{m}]_{n}\left(\bu_{\fs} \right)=\Xi_{\fs} ,  \]
where $t^{m}*M_{f}$ is the Frobenius module defined by \[t^{m}* \Phi_{f}:=\left(
     \begin{array}{cc}
       (t-\theta)^{n} & 0 \\
       t^{m}f^{(-1)}(t-\theta)^{n} & 1 \\
     \end{array}
   \right) \] (see~\cite[\S\S~2.4]{CPY14}). It follows that $\Phi_{g}$ and $t^{m}*\Phi_{f}$ define the same class of Frobenius modules in $\Ext_{\cF}^{1}\left( {\bf{1}},C^{\otimes n} \right)$ since the classes of their defining Frobenius modules are mapped to same point $\Xi_{\fs}$ under $\Delta_{n}$. In other words, there exists a polynomial $h\in \ok[t]$ satisfying the following Frobenius difference equations
   \[  \left(
         \begin{array}{cc}
           1 & 0 \\
           h & 1 \\
         \end{array}
       \right)^{(-1)} \left(
                        \begin{array}{cc}
                          (t-\theta)^{n} & 0 \\
                          g^{(-1)}(t-\theta)^{n} & 1 \\
                        \end{array}
                      \right)
       = \left(
           \begin{array}{cc}
             (t-\theta)^{n} & 0 \\
             t^{m}f^{(-1)}(t-\theta)^{n} & 1 \\
           \end{array}
         \right)
        \left(
                                                  \begin{array}{cc}
                                                    1 & 0 \\
                                                    h & 1 \\
                                                  \end{array}
                                                \right)
       ,\] from which we derive the following identity
 \begin{equation}\label{E:DiffE g tmg}
   \begin{array}{rl}
      & \left(
           \begin{array}{cc}
             I_{r} &  \\
             h^{(-1)},0,\ldots,0 & 1 \\
           \end{array}
         \right)
          \left(
               \begin{array}{cc}
                 \Phi' &  \\
              g^{(-1)}(t-\theta)^{n},0,\ldots, 0   & 1 \\
               \end{array}
             \right)\\
     = &  \left(
               \begin{array}{cc}
                 \Phi' &  \\
             t^{m} f^{(-1)}(t-\theta)^{n},0,\ldots, 0   & 1 \\
               \end{array}
             \right) \left(
           \begin{array}{cc}
             I_{r} &  \\
             h,0,\ldots,0 & 1 \\
           \end{array}
         \right), \\
   \end{array}
 \end{equation} where $\Phi':=\Phi_{\fs}'$ is given in (\ref{E:Phi s'}).

We define the following two matrices
\[\widetilde{\Phi}_{g}:= \left(
               \begin{array}{cc}
                 \Phi' &  \\
              g^{(-1)}(t-\theta)^{n},0,\ldots, 0   & 1 \\
               \end{array}
             \right)\hbox{ and }\widetilde{\Phi}_{t^{m}f}:= \left(
               \begin{array}{cc}
                 \Phi' &  \\
             t^{m} f^{(-1)}(t-\theta)^{n},0,\ldots, 0   & 1 \\
               \end{array}
             \right). \]
Let $\Psi':=  \Psi_{\fs}'\in \GL_{r}(\TT)$ (resp. $\Psi:= \Psi_{\fs}\in \GL_{r+1}(\TT)$) be given in (\ref{E:Psi'}) (resp. (\ref{E:Psi})) and note that $\Psi$ is of the form
\[
\Psi=\left(
       \begin{array}{cc}
         \Psi' &  \\
         \mu & 1 \\
       \end{array}
     \right)
 ,
\] where
\[ \mu=\left(\cL_{(s_{1},\ldots,s_{r})},\ldots,\cL_{s_{r}}  \right)   . \]
We further put
\[ \widetilde{\Psi}_{t^{m}f}:=  \left(
               \begin{array}{cc}
                 \Psi' & \\
              t^{m}\Omega^{n}\cL_{f},0,\ldots, 0   & 1 \\
               \end{array}
             \right) \] and note that using (\ref{E:DiffPsiPsi'}) we have
             \[   \widetilde{\Psi}_{t^{m}f}^{(-1)}= \widetilde{\Phi}_{t^{m}f} \widetilde{\Psi}_{t^{m}f}   .\]

{\bf{Claim}}: There exists a matrix $\nu$ of the form
\[  \nu=\left(
          \begin{array}{cc}
            I_{r} &  \\
          \nu_{1},\ldots,\nu_{r} & 1 \\
          \end{array}
        \right)\in \GL_{r+1}(\ok[t])
  \] so that
\[  \nu^{(-1)} \left( \alpha *\Phi \right)=\widetilde{\Phi}_{t^{m}f}\cdot \nu , \]
where \[\Phi:=\Phi_{\fs}:=\left(
                            \begin{array}{cc}
                              \Phi' &  \\
                              \varphi & 1 \\
                            \end{array}
                          \right)
  \hbox{ given in }\eqref{E:Phi s},\] and
  \[ \alpha*\Phi:=\left(
                    \begin{array}{cc}
                      \Phi' &  \\
                      \alpha \varphi & 1 \\
                    \end{array}
                  \right)
    .\]

We first assume the claim above to finish the proof. Define
\[ \alpha*\Psi:=\left(
                  \begin{array}{cc}
                    \Psi' &  \\
                    \alpha \mu & 1 \\
                  \end{array}
                \right)\in \GL_{m+1}(\TT).
    \] Since $\alpha\in \FF_{q}[t]$, using (\ref{E:DiffPsiPsi'}) we have
    \[ \left( \alpha*\Psi\right)^{(-1)}= \alpha*\Phi \cdot \alpha*\Psi   .\]
Note that the claim above implies
\[ \left( \nu \cdot \alpha*\Psi \right)^{(-1)} =\nu^{(-1)}\cdot \alpha*\Phi \cdot \alpha*\Psi=\widetilde{\Phi}_{t^{m}f}\left( \nu\cdot \alpha*\Psi    \right).   \]
In other words,  $\nu\cdot \alpha*\Psi  $ is also a {\it{fundamental matrix}} for $\widetilde{\Phi}_{t^{m}f}$ in the sense of \cite[\S\S~4.1.6]{P08} and hence by \cite[\S~4.1.6]{P08} that there exists a matrix $\gamma\in \GL_{r+1}(\FF_{q}(t))$ of the form
\[ \gamma=\left(
            \begin{array}{cc}
              I_{r} &  \\
              \gamma_{1},\ldots,\gamma_{r} & 1 \\
            \end{array}
          \right)
  \] so that
  \[
    \nu \cdot \alpha*\Psi=\widetilde{\Psi}_{t^{m}f}\cdot \gamma .\]
By comparing with the $(r+1,1)$-entries of both sides of the equation above we obtain the following identity
\begin{equation}\label{E:Key Identity}
 \nu_{1}\Omega^{n}+\nu_{2}\Omega^{s_{2}+\cdots+s_{r}}\cL_{s_{1}}+\cdots+\nu_{r}\Omega^{s_{r}}\cL_{(s_{1},\ldots,s_{r-1})}+\alpha\cL_{(s_{1},\ldots,s_{r})}=t^{m} \Omega^{n}\cL_{f}+\gamma_{1}.
\end{equation}
By Lemma~\ref{L:L theta qN} we have that for each $N\in \ZZ_{\geq 0}$
\[ \cL_{(s_{1},\ldots,s_{r})}(\theta^{q^{N}})=\cL_{(s_{1},\ldots,s_{r})}(\theta)^{q^{N}}=\left( \Gamma_{s_{1}}\cdots\Gamma_{s_{r}}\zeta_{A}(s_{1},\ldots,s_{r})/ \tilde{\pi}^{n}  \right)^{q^{N}}    ,\]
and by Lemma~\ref{L:LogPeriod} and \cite[Prop.~2.3.3]{CPY14} we have that for each $N\in \ZZ_{\geq 0}$,
\[ \left(\Omega^{n}\cL_{f}\right)(\theta^{q^{N}})= \left(\Omega^{n}\cL_{f}\right)(\theta)^{q^{N}}=\left( y/\tilde{\pi}^{n}\right)^{q^{N}} ,\] where $y$ is the last coordinate of $\log_{n}(\bu_{\fs})$ as $f$ is the polynomial associated to $\bu_{\fs}$. We mention that $\gamma_{1}$ must be in $\FF_{q}[t]$ since all other terms of \eqref{E:Key Identity} are in the Tate algebra $\TT$. Note that $\Omega$ has a simple zero at $t=\theta^{q^{N}}$ for each $N\in \NN$ and hence by Lemma~\ref{L:L theta qN}
\[  \Omega^{s_{2}+\cdots+s_{r}}\cL_{(s_{1})} (\theta^{q^{N}})=\cdots=\Omega^{s_{r}}\cL_{(s_{1},\ldots,s_{r-1})} (\theta^{q^{N}})=0    .\]
It follows that  specializing (\ref{E:Key Identity}) at $t=\theta^{q^{N}}$ for a positive integer $N$ and taking the $q^{N}$-th root we obtain the following identity
\[ \alpha(\theta)\Gamma_{s_{1}}\cdots\Gamma_{s_{r}}\zeta_{A}(\fs)=\theta^{m}y+\gamma_{1}(\theta)\tilde{\pi}^{n}    .\]
Now we put
\[ Y_{\fs}:=\partial[t^{m}]_{n}\cdot \log_{n}(\bu_{\fs})+\partial[\gamma_{1}]_{n} \lambda_{n}   \]
and note that the last coordinate of $Y_{\fs}$ is $ \theta^{m}y+\gamma_{1}(\theta)\tilde{\pi}^{n}=\alpha(\theta)\Gamma_{s_{1}}\cdots\Gamma_{s_{r}}\zeta_{A}(\fs)$. Since $\partial[\gamma_{1}]_{n} \lambda_{n}\in \Lambda_{n}$, by the functional equation of $\exp_{n}$ we see that $\exp_{n}(Y_{\fs})=[t^{m}]_{n}\left( \bu_{\fs}\right) =\Xi_{\fs}$, whence completing the proof of Theorem~\ref{T1:LogMZV}.

The rest task is to prove the claim above. We recall that $M'$ is the Frobenius module defined by the matrix $\Phi'$ with respect to a $\ok[t]$-basis $\left\{ m_{1},\ldots,m_{r}\right\}$ of $M'$, and we have the following isomorphism as $\FF_{q}[t]$-modules
\[ \Ext_{\cF}^{1}\left({\bf{1}},M'  \right)\cong E_{\fs}'(\ok)  \]
and note that $[M]$ is mapped to $\bv_{\fs}$. Hence $[\alpha* M]$ is mapped to \[\Xi_{\fs}:=\rho_{\alpha}( \bv_{\fs})\in \bC^{\otimes n}(A)\hookrightarrow E_{\fs}'(A).\]

We note that $M'$ is a free left $\ok[\sigma]$-module with a natural $\ok[\sigma]$-basis:
 \begin{equation}\label{E:basis}
 \left\{(t-\theta)^{s_{1}+\cdots+s_{r}-1}m_{1},\cdots,(t-\theta)m_{1},m_{1},\ldots,(t-\theta)^{s_{r}-1}m_{r},\ldots,(t-\theta)m_{r},m_{r} \right\}
 \end{equation}
 (see \cite[Proof of Thm.~5.2.1]{CPY14}). We recall that $g\in A[t]$ is the polynomial associated to $\Xi_{\fs}$ and hence via the isomorphism $\Ext_{\cF}^{1}\left({\bf{1}},M'  \right)\cong E_{\fs}'(\ok)$, the class of the Frobenius module defined by the matrix $\widetilde{\Phi}_{g}$ is mapped to $\Xi_{\fs}\in \bC^{\otimes n}(A)\hookrightarrow E_{\fs}'(A)$ via the basis $(\ref{E:basis})$ (see~\cite[(5.2.2)]{CPY14}). In other words, the two matrices $\widetilde{\Phi}_{g}$ and
\[ \alpha* \Phi=\left(
                    \begin{array}{cc}
                      \Phi' & {\bf{0}} \\
                    0,\ldots,\alpha H_{s_{r}-1}^{(-1)}(t-\theta)^{s_{r}} & 1 \\
                    \end{array}
                  \right)
   \] define the same class of Frobenius modules in $\Ext_{\cF}^{1}\left({\bf{1}},M'  \right)$. Therefore there exists a matrix
   \[ \delta=\left(
               \begin{array}{cc}
                 I_{r} &  \\
                 \delta_{1},\ldots,\delta_{r} & 1 \\
               \end{array}
             \right)\in \GL_{r+1}(\ok[t])
     \] so that
     \begin{equation}\label{E:DiffE delta Phig}
      \delta^{(-1)} \cdot \alpha*\Phi=\widetilde{\Phi}_{g}\cdot \delta  .
     \end{equation}

Put \[\eta:=\left(
              \begin{array}{cc}
                I_{r} & {\bf{0}} \\
                h,0,\ldots,0 & 1 \\
              \end{array}
            \right)\in \GL_{r+1}(\ok[t])
  \] and note that by (\ref{E:DiffE g tmg}) we have $\eta^{(-1)}\cdot \widetilde{\Phi}_{g}=\widetilde{\Phi}_{t^{m}f}\cdot \eta$. Putting $\nu:=\eta \delta\in \GL_{r+1}(\ok[t])$ and using (\ref{E:DiffE delta Phig}) and (\ref{E:DiffE g tmg}) we have the desired identity
  \[\nu^{(-1)}\cdot \alpha*\Phi= \eta^{(-1)}\delta^{(-1)} \cdot \alpha*\Phi =\eta^{(-1)}\cdot \widetilde{\Phi}_{g}\cdot \delta= \widetilde{\Phi}_{t^{m}f}\cdot \eta \delta =\widetilde{\Phi}_{t^{m}f} \cdot \nu. \] From the explicit forms of $\delta$ and $\eta$ we see that $\nu$ has the desired form, whence proving the claim above.

\section{Step~III: The Identity}
\subsection{The dimension formula}
In this section, we will give a  proof of Theorem~\ref{T:MainThmIntroduction}~(1). Combining Theorems~\ref{T1:LogMZV} and \ref{T:Yu's thm}, we first prove the following.
\begin{theorem}\label{T:DimHigherDepth}
Let $n\geq 2$ be an integer and let $\bv_{n}\in \bC^{\otimes n}(A)$ be the special point given in Theorem~\ref{T:AT90}. For $\fs=(s_{1},\ldots,s_{r})\in \NN^{r}$ with $r\geq 2$, we denote by $\fs':=(s_{2},\ldots,s_{r})$.  Put
\[\mathcal{E}_{n}:=\left\{\fs=(s_{1},\ldots,s_{r})\in \NN^{r}; r\geq 2, \wt(\fs)=n \hbox{ and }\zeta_{A}(\fs')\hbox{ is Eulerian}    \right\}.   \]
For each $\fs\in \mathcal{E}_{n}$, let $\Xi_{\fs}\in \bC^{\otimes n}(A)$ be given in Theorem~\ref{T1:LogMZV}. Then we have
\[ \dim_{k}\Sp_{k}\left\{ \tilde{\pi}^{n},\zeta_{A}(n),\zeta_{A}(\fs);\fs\in \mathcal{E}_{n}\right\} = 1+\rk_{\FF_{q}[t]} \Sp_{\FF_{q}[t]}\left\{\bv_{n}\right\}\cup_{\fs\in \mathcal{E}_{n}}\left\{ \Xi_{\fs} \right\}     .\]
\end{theorem}

\begin{proof}  To prove the theorem, it suffices to prove the following equivalent statements:
\[
  \begin{array}{rl}
     & \left\{\tilde{\pi}^{n},\zeta_{A}(n), \zeta_{A}(\fs);\fs\in \mathcal{E}_{n}  \right\}{\hbox{ are linearly dependent over }}k    \\
& \\
  \Leftrightarrow   &  \left\{ \bv_{n}\right\}\cup_{\fs\in \mathcal{E}_{n}} \left\{\Xi_{\fs} \right\} {\hbox{ are linearly dependent over }}\FF_{q}[t].     \\
  \end{array}
\]

Proof of ($\Rightarrow$). Suppose that there exist polynomials (not all zero) $\left\{\eta \right\}\cup  \left\{\eta_{\fs} \right\}_{\fs\in \mathcal{E}_{n}} \subseteq \FF_{q}[t]$ for which \[[\eta]_{n}(\bv_{n})+ \sum_{\fs\in \mathcal{E}_{n}}[\eta_{\fs}]_{n} \left( \Xi_{\fs}\right)={\bf{0}}.\] By Theorem~\ref{T:AT90}, there exists a vector $Y_{n}$ of the form
 \[ Y_{n}=\left(
            \begin{array}{c}
              * \\
              \vdots \\
              \Gamma_{n}\zeta_{A}(n) \\
            \end{array}
          \right)\in \CC_{\infty}^{n}
    \]
 for which $\exp_{n}(Y_{n})=\bv_{n}$.   For each $\fs\in \mathcal{E}_{n}$, by Theorem~\ref{T1:LogMZV}  there exists vectors $Y_{\fs}\in \CC_{\infty}^{n}$  satisfying the property in Theorem~\ref{T1:LogMZV}. We define the vector
\[ Y:=\partial[\eta]_{n}Y_{n}+ \sum_{\fs\in\mathcal{E}_{n}}\partial[\eta_{\fs}]_{n} Y_{\fs} \]
and note that
\[\exp_{n}(Y)=[\eta]_{n}\left(\exp_{n}(Y_{n})\right) +   \sum_{\fs\in \mathcal{E}_{n}}[\eta_{\fs}]_{n}\left( \exp_{n}(Y_{\fs})\right)=[\eta]_{n}(\bv_{n}) +\sum_{\fs\in \mathcal{E}_{n}} [\eta_{\fs}]_{n}\left( \Xi_{\fs}\right)={\bf{0}}, \]
whence
\[ Y\in \Ker \exp_{n}=\partial[\FF_{q}[t]]_{n} \lambda_{n}   .\]
Note that the last coordinate of $\lambda_{n}$ is $\tilde{\pi}^{n}$ and that for each $a\in\FF_{q}[t]$, $\partial[a]_{n}$ is an upper triangular matrix with $a(\theta)$ down the diagonals. Taking the last coordinates from both sides of the equality above gives the desired result.

Proof of ($\Leftarrow$). We suppose that there exist polynomials $\left\{ \delta_{0},\delta, \delta_{\fs};\fs\in \mathcal{E}_{n} \right\}\in A$ (not all zero) so that
\[ \delta_{0}\tilde{\pi}^{n}+\delta \zeta_{A}(n)  +\sum_{\fs\in \mathcal{E}_{n}}\delta_{\fs}\zeta_{A}(\fs)   =0,\] which can be also written as
\[\delta_{0}\tilde{\pi}^{n} + \frac{\delta}{\Gamma_{n}}\Gamma_{n}\zeta_{A}(n)  +\sum_{\fs\in \mathcal{E}_{n}} \frac{\delta_{\fs}}{\alpha_{\fs}(\theta)\Gamma_{\fs} }\alpha_{\fs}(\theta)\Gamma_{\fs}\zeta_{A}(\fs)    =0 ,\] where $\Gamma_{\fs}:=\Gamma_{s_{1}}\ldots\Gamma_{s_{r}}$ for $\fs=(s_{1},\ldots,s_{r})$ and $\alpha_{\fs}\in \FF_{q}[t]$ is given in Theorem~\ref{T1:LogMZV}.   Multiplying a common denominator of the coefficients of the equation above shows that there exist polynomials $\left\{\eta_{0},\eta, \zeta_{\fs};\fs\in \mathcal{E}_{n}  \right\}\subseteq   \FF_{q}[t]$ (not all zero) so that
\[  \eta_{0}(\theta)\tilde{\pi}^{n}+\eta(\theta)\Gamma_{n}\zeta_{A}(n)  +\sum_{\fs \in \mathcal{E}_{n}}\eta_{\fs}(\theta) \alpha_{\fs}\Gamma_{\fs}\zeta_{A}(\fs)=0  .\]
For each $\fs\in \mathcal{E}_{n}$, let $Y_{\fs}\in \CC_{\infty}^{n}$ be given in Theorem~\ref{T1:LogMZV} and note that its last coordinate is $\alpha_{\fs}(\theta)\Gamma_{\fs}\zeta_{A}(\fs)$. So the last coordinate of \[Y:= \partial[\eta_{0}]_{n}\lambda_{n}+\partial[\eta]_{n}Y_{n} +\sum_{\fs\in \mathcal{E}_{n}}\partial[\eta_{\fs}]_{n} Y_{\fs}\]
is zero by the equation above. Since
\[\exp_{n}(Y)=\exp_{n}\left( \partial[\eta_{0}]_{n}\lambda_{n}+\partial[\eta]_{n}Y_{n} +\sum_{\fs\in \mathcal{E}_{n}}\partial[\eta_{\fs}]_{n} Y_{\fs}    \right)=[\eta]_{n}\left(\bv_{n} \right)+  \sum_{\fs\in \mathcal{E}_{n}}[\eta_{\fs}]_{n} \left(\Xi_{\fs} \right)\in \bC^{\otimes n}(A),
\] Theorem~\ref{T:Yu's thm} implies that $Y$ has to be zero, and hence
\[ [\eta]_{n}\left(\bv_{n} \right)+ \sum_{\fs\in \mathcal{E}_{n}}[\eta_{\fs}]_{n} \left(\Xi_{\fs} \right)={\bf{0}}  .\]
\end{proof}

\begin{remark}
For a positive integer $n$, we recall that $\zeta_{A}(n)/\tilde{\pi}^{n}\in k$ for  $A$-even $n$ by \cite{Ca35} and in which case $\bv_{n}$ is an $\FF_{q}[t]$-torsion point by Remark~\ref{Rem:TorsionEven}. When $n$ is $A$-odd, we have $\tilde{\pi}^{n}\notin k_{\infty}$ and so $\tilde{\pi}^{n}$ is $k$-linearly independent from MZV's, whence from the proof above we derive that
\[[\eta]_{n}(\bv_{n})+\sum_{\fs\in \mathcal{E}_{n}}[\eta_{\fs}]_{n} (\Xi_{\fs})={\bf{0}}\]
if and only if
\[ \eta(\theta)\Gamma_{n}\zeta_{A}(n)+\sum_{\fs\in \mathcal{E}_{n}}\eta_{\fs}(\theta)\Gamma_{\fs}\zeta_{A}(\fs)=0   \]
for $\left\{\eta\right\}\cup_{\fs\in \mathcal{E}_{n}}\left\{\eta_{\fs} \right\}\subseteq \FF_{q}[t] $.
\end{remark}

By this remark, we immediately obtain the following two consequences.
\begin{corollary}\label{C:DimHigherDepth}
For an integer $n \geq2$, we continue with the notation in Theorem~\ref{T:DimHigherDepth}. Then we have
 \[
   \begin{array}{rl}
      &
 \dim_{k} \Sp_{k}\left\{\zeta_{A}(n),\zeta_{A}(\fs);\fs\in \mathcal{E}_{n}\right\}  \\
      &  \\
     = & \begin{cases}
     & 1+\rk_{\FF_{q}[t]} \Sp_{\FF_{q}[t]}\left\{ \Xi_{\fs}; \fs\in \mathcal{E}_{n} \right\}  \hbox{ if }n\hbox{ is }A\hbox{-even},  \\
     &\rk_{\FF_{q}[t]} \Sp_{\FF_{q}[t]}\left\{\bv_{n}, \Xi_{\fs}; \fs\in \mathcal{E}_{n} \right\} \hbox{ if }n\hbox{ is }A\hbox{-odd}.  \\
     \end{cases} \\
   \end{array}
\]

\end{corollary}

\begin{corollary}\label{T:DimZetaEven}
Let $n\geq 2$ be a positive integer. Put
\[ \mathcal{V}:=\left\{ (s_{1},s_{2})\in \NN^{2}; s_{1}+s_{2}=n\hbox{ and }(q-1)|s_{2}  \right\}  .\] For each $\fs\in \mathcal{V}$, let $\Xi_{\fs}$ be given in Theorem~\ref{T1:LogMZV}. Then we have
\[ \dim_{k}\Sp_{k}\left\{\tilde{\pi}^{n},\zeta_{A}(\fs);\fs\in \mathcal{V}  \right\} = 1+\rk_{\FF_{q}[t]}\Sp_{\FF_{q}[t]}\left\{\Xi_{\fs} \right\}_{\fs\in \mathcal{V}}    .\]
\end{corollary}

\subsection{Proof of Theorem~\ref{T:MainThmIntroduction}~(1)} Here we give a proof for part (1) of Theorem~\ref{T:MainThmIntroduction}, which is addressed as the following result.

\begin{corollary}\label{C:ReMainThm1}
Let $n\geq 2$ be an integer. Put
\[ \mathcal{S}:=\left\{ \tilde{\pi}^{n},\zeta_{A}(1,n-1),\zeta_{A}(2,n-2),\ldots,\zeta_{A}(n-1,1) \right\} ,\]
and
\[ \mathcal{V}:=\left\{ (s_{1},s_{2})\in \NN^{2}; s_{1}+s_{2}=n\hbox{ and }(q-1)|s_{2}  \right\}  .\] For each $\fs\in \mathcal{V}$, let $\Xi_{\fs}$ be given in Theorem~\ref{T1:LogMZV}. Then we have
\[ \dim_{k} \Sp_{k}\mathcal{S}  =n- \lfloor \frac{n-1}{q-1} \rfloor+\rank_{\FF_{q}[t]} \Sp_{\FF_{q}[t]}\left\{\Xi_{\fs} \right\}_{\fs\in \mathcal{V}}    .\]
\end{corollary}
\begin{proof}
Note that we have the following equalities
\[
     \begin{array}{rl}
      \dim_{k} \Sp_{k}\mathcal{S} = & |\mathcal{S}\backslash \left\{\tilde{\pi}^{n}, \zeta_{A}(\fs); \fs\in \mathcal{V}  \right\}| +\dim_{k} \Sp_{k}\left\{ \tilde{\pi}^{n},\zeta_{A}(\fs);\fs\in \mathcal{V} \right\} \\
       = & \left(n-1-\lfloor \frac{n-1}{q-1} \rfloor \right)+\left(1+\rank_{\FF_{q}[t]} \Sp_{\FF_{q}[t]}\left\{\Xi_{\fs} \right\}_{\fs\in \mathcal{V}}  \right) \\
       = &n- \lfloor \frac{n-1}{q-1} \rfloor+\rank_{\FF_{q}[t]} \Sp_{\FF_{q}[t]}\left\{\Xi_{\fs} \right\}_{\fs\in \mathcal{V}},  \\
     \end{array}
   \]where the first equality comes from Corollary~\ref{C:NecCondDepth2}, and the second equality comes from Corollary~\ref{T:DimZetaEven}.
\end{proof}

\begin{remark}
Recently, some algebraic independence results of certain MZV's are obtained by Mishiba~\cite{M14}, but the coordinates of those MZV's are restricted to be $A$-odd with other hypotheses. Concerning this issue, we refer the reader to \cite{M14}.
\end{remark}

\section{Step~IV: A Siegel's Lemma}
\subsection{The main result}
The primary result in this section is the following theorem, which implies Theorem~\ref{T:MainThmIntroduction}~(2) and so allows us to compute the exact quantity in Theorem~\ref{T:MainThmIntroduction}~(1).
\begin{theorem}\label{T:Siegel} Let $n$ be a positive integer and $\bv_{1},\ldots,\bv_{m}\in \bC^{\otimes n}(A)$. Then we have an effective algorithm to compute the following rank
\[  r_{m}:=\rk_{\FF_{q}[t]} \Sp_{\FF_{q}[t]}\left\{ \bv_{1},\ldots,\bv_{m} \right\}   .\]
\end{theorem}
\begin{proof}
We let $\bv_{i}:=(v_{i1},\ldots,v_{in})^{\tr}\in \bC^{\otimes n}(A)$, and let $f_{i}:=v_{i1}(t-\theta)^{n-1}+\cdots+v_{in}\in A[t]$ be its associated polynomial. Note that the class $[M_{i}]\in \Ext_{\cF}^{1}\left({\bf{1}},C^{\otimes n}\right)\cong \bC^{\otimes n}(\ok) $ is mapped to $\bv_{i}$, where $M_{i}\in \cF$ is defined by the matrix
\[ \left(
     \begin{array}{cc}
       (t-\theta)^{n} & 0 \\
       f_{i}^{(-1)}(t-\theta)^{n} & 1 \\
     \end{array}
   \right)
 .\]

Fix $m$ polynomials $a_{1},\ldots,a_{m}\in \FF_{q}[t]$ and put $F:=\sum_{i=1}^{m}a_{i}f_{i}$. Then the class  $[M_{F}]\in \Ext_{\cF}^{1}\left({\bf{1}},C^{\otimes n} \right)$ is mapped to the integral point $\sum_{i=1}^{m}[a_{i}]_{n}\bv_{i}$ (see~\cite[\S\S~2.4]{CPY14}), where $M_{F}\in \cF$ is defined by the matrix
\[ \left(
     \begin{array}{cc}
       (t-\theta)^{n} & 0 \\
       F^{(-1)}(t-\theta)^{n} & 1 \\
     \end{array}
   \right)
   .\] It follows that $\sum_{i=1}^{m}[a_{i}]_{n}\bv_{i}={\bf{0}}$ if and only if $M_{F}$ presents the trivial class in $\Ext_{\cF}^{1}\left({\bf{1}},\bC^{\otimes n} \right)$. Therefore, we have the following equivalence:
\begin{enumerate}
\item[$\bullet$]  $\sum_{i=1}^{m}[a_{i}]_{n}\bv_{i}={\bf{0}}\in \bC^{\otimes n}(A)$.
\item[$\bullet$] there exists a polynomial $\delta\in \ok[t]$ for which
\[ \left(
     \begin{array}{cc}
       1 & 0 \\
       \delta & 1 \\
     \end{array}
   \right)^{(-1)}\left(
                   \begin{array}{cc}
                     (t-\theta)^{n} & 0 \\
                     F^{(-1)}(t-\theta)^{n} & 1 \\
                   \end{array}
                 \right)= \left(
                            \begin{array}{cc}
                              (t-\theta)^{n} & 0 \\
                              0 & 1 \\
                            \end{array}
                          \right) \left(
                                    \begin{array}{cc}
                                      1 & 0 \\
                                      \delta & 1 \\
                                    \end{array}
                                  \right),
   \] which is equivalent to
   \begin{equation}\label{E:deltaF}
   \delta^{(-1)}(t-\theta)^{n}+F^{(-1)}(t-\theta)^{n}=\delta.
   \end{equation}

\end{enumerate}

$\underline{\bf{Step~I}}$. Assume that there exist $a_{1},\ldots,a_{m}\in \FF_{q}[t]$ for which $\sum_{i=1}^{m}[a_{i}]_{n}\bv_{i}={\bf{0}}$, ie., the equation (\ref{E:deltaF}) holds. Then the $\delta$ in (\ref{E:deltaF}) must be in $A[t]$.

$\underline{\hbox{Proof of Step~I}}$. Note that the equation (\ref{E:deltaF}) is equivalent to $\delta (t-\theta^{q})^{n}+F (t-\theta^{q})^{n}=\delta^{(1)}$. Then the result follows from H.-J. Chen's formulation in the proof of \cite[Thm.~2~(a)]{KL15}.

For a polynomial $h=\sum_{i}u_{i}t^{i}\in A[t]$, we define its sup-norm by $\parallel h\parallel:=\max_{i}\left\{|u_{i}|_{\infty} \right\}$. Note that for $h_{1},h_{2}\in A[t]$ we have:
\begin{enumerate}
\item[$\bullet$] $\parallel h_{1}h_{2}\parallel=\parallel h_{1}\parallel \cdot \parallel h_{2}\parallel$.
\item[$\bullet$] $\parallel h_{1}+h_{2}\parallel \leq  \max \left\{\parallel h_{1}\parallel, \parallel h_{2}\parallel\right\} $.
\end{enumerate}
For each $\bv_{i}$, we define $\parallel \bv_{i}\parallel:={\rm{max}}_{1\leq j\leq n}\left\{ |v_{ij}|_{\infty} \right\}$, and note that
\[ \parallel f_{i}\parallel\leq \parallel\bv_{i}\parallel \cdot |\theta|_{\infty}^{n-1}  .\]
Put \[D:={\rm{max}}_{1\leq i \leq m}\left\{ \parallel\bv_{i}\parallel\right\}\cdot |\theta|_{\infty}^{n-1} \]
and so \[\parallel F\parallel\leq D  .\]

$\underline{\bf{Step~II}}$. Let hypotheses be given as in {\bf{Step~I}}. Define
\[ \ell:={\rm{max}}\left\{ \log_{|\theta|_{\infty}}D+1,\frac{nq}{q-1}+1   \right\} .\] Then $\deg_{\theta}\delta< \ell$ when we regard $\delta$ as a polynomial in the variable ${\theta}$ over $\FF_{q}[t]$.

$\underline{\hbox{Proof of Step~II}}$. Suppose on the contrary that $\deg_{\theta}\delta \geq \ell$, ie., $\parallel \delta\parallel\geq |\theta|_{\infty}^{\ell}$. Note that by the definition of $\ell$ we have
\begin{equation}\label{E:equa1}
\parallel F(t-\theta^{q})^{n}\parallel< |\theta|_{\infty}^{\ell}\cdot |\theta|_{\infty}^{qn}\leq \parallel \delta (t-\theta^{q})^{n}\parallel.
\end{equation}
Therefore, \eqref{E:equa1} and the equality
\[ \delta (t-\theta^{q})^{n}+F(t-\theta^{q})^{n}=\delta^{(1)}  \]
imply that
\[ \parallel \delta(t-\theta^{q})^{n} \parallel =\parallel \delta^{(1)}\parallel .\]
In other words, we have that
\[  \deg_{\theta} \delta+ nq= q \deg_{\theta}\delta, \]
whence
\[  \deg_{\theta}\delta=\frac{nq}{q-1}<\ell,   \]
a contradiction.

$\underline{\hbox{\bf{Step~III: End of proof}}}$. Now we write
\begin{equation}\label{E:formdeltaF}
 \delta:=c_{1}\theta^{\ell-1}+\cdots+c_{\ell}\in \FF_{q}[t][\theta] \hbox{ and }F=d_{1}\theta^{\ell-1}+\cdots+d_{\ell}\in \FF_{q}[t][\theta]
 \end{equation}
and note that the coefficients $d_{1},\ldots,d_{\ell}$ are $\FF_{q}[t]$-linear combinations of $a_{1},\ldots,a_{m}$. We recall that solving $a_{1},\ldots,a_{m}$ in the equation
\[\sum_{i=1}^{m} [a_{i}]_{n}(\bv_{i})={\bf{0}} \]
is equivalent to solving for $\delta$ and $a_{1},\ldots,a_{m}$ satisfying
\[ \delta(t-\theta^{q})^{n}+F(t-\theta^{q})^{n} =\delta^{(1)}   .\]
However, putting the forms of $\delta$ and $F$ (\ref{E:formdeltaF}) into the equation above and comparing the coefficients of each $\theta^{i}$ for $i=0,\ldots,\ell-1$ we obtain a system of linear equations in $c_{1},\ldots,c_{\ell},a_{1},\ldots,a_{m}$ over $\FF_{q}[t]$. Using Gauss elimination we can solve for solutions $c_{1},\ldots,c_{\ell},a_{1},\ldots,a_{m}$ effectively, and particularly obtain the rank of the solutions $a_{1},\ldots,a_{m}$, whence establishing the desired result.
\end{proof}

\begin{remark}
We mention that in \cite{De91, De92} Denis studied the question of Siegel's lemma type for $t$-modules. For integral points $\bv_{1},\ldots,\bv_{m}\in \bC^{\otimes n}(A)$, Denis showed that there exists a constant $c$ (depending on $n$ and $\bv_{1},\ldots,\bv_{m}$) so that the degrees of the coefficients of any $\FF_q[t]$-linear relations among $\bv_{1},\ldots,\bv_{m}$ can be bounded by $c$. However, the value of $c$ is not explicit in Denis' results, and  our approach is entirely different from his.
\end{remark}

\subsection{The algorithm}\label{sec:algorithm} For weight $n\geq 2$, we provide the following algorithm to compute the dimension
\[ d_{n}:=\dim_{k} \Sp_{k}\left\{ \tilde{\pi}^{n},\zeta_{A}(1,n-1),\zeta_{A}(2,n-2),\ldots,\zeta_{A}(n-1,1)  \right\}     .\]
Let $\mathcal{V}$ be given in Theorem~\ref{T:MainThmIntroduction}. The algorithm is basically divided into two parts (most of the first part was given in \cite[\S~6.1.1]{CPY14}).

(I) $\underline{\hbox{\bf{Computing the integral points }}\Xi_{\fs}}$: Fix an $\fs=(s_{1},s_{2})\in \mathcal{V}$.
\begin{enumerate}
\item[I-1] Compute the Anderson-Thakur polynomials $H_{s_{1}-1},H_{s_{2}-1}$.

\item[I-2] Let $M_{\fs}'$ be the Frobenius module defined by $\Phi_{\fs}'$ as in  (\ref{E:Phi s'}) with $\ok[t]$-basis $\left\{ m_{1},m_{2} \right\}$. Put $d=(s_{1}+s_{2})+s_{2}$ and let $\left\{\nu_{1},\ldots,\nu_{d} \right\}$ be the $\ok[\sigma]$-basis of $M_{\fs}'$ given by \[(t-\theta)^{s_{1}+s_{2}-1}m_{1},\dots,(t-\theta)m_{1},m_{1},(t-\theta)^{s_{2}-1}m_{2},\ldots,(t-\theta)m_{2},m_{2} .\] Identify $M_{\fs}'/(\sigma-1)M_{\fs}'$ with $\Mat_{d\times 1}(\ok)$ via $\nu_{1},\ldots,\nu_{d}$ in \cite[(5.2.2)]{CPY14}.

\item[I-3] Write down the $t$-action on $M_{\fs}'/(\sigma-1)M_{\fs}'$, and so give a $t$-module structure on $\Mat_{d\times 1}(\ok)$, which we denote by $(E_{\fs}',\rho)$.

\item[I-4] Consider $H_{s_{2}-1}^{(-1)}(t-\theta)^{s_{2}}m_{2}\in M_{\fs}'/(\sigma-1)M_{\fs}'$, which corresponds to an integral point $\bv_{\fs}=(a_{1},\ldots,a_{d})^{\tr}\in E'(A)$ from the decomposition $H_{s_{2}-1}^{(-1)}(t-\theta)^{s_{2}}m_{r}\equiv\sum_{i=1}^{d} a_{i} \nu_{i}$ (mod $\sigma-1$). See \cite[\S\S~5.2]{CPY14}.

\item[I-5] Decompose $s_{2} = p^{\ell} n_1\left(q^h-1\right)$ where $p\nmid n_1$ and $h$ is the greatest integer such that $(q^h-1) \mid s_{2}$. Define the polynomial $\alpha_{\fs}:=(t^{q^{h}}-t)^{p^{\ell}}\in \FF_{q}[t]$, and then compute $\Xi_{\fs}:= \rho_{\alpha_{\fs}}(\bv_{\fs})$, which is identified in $\bC^{\otimes n}(A)\hookrightarrow E_{\fs}'(A)$.
\end{enumerate}

(II) $\underline{\hbox{\bf{Computing the dimension }}d_{n}}$:
\begin{enumerate}
\item[II-1] For each $\fs\in \mathcal{V}$, define $\parallel \Xi_{\fs}\parallel$ to the maximum of the absolute values of components of $\Xi_{\fs}$ and put $D:={\rm{max}}\left\{ \parallel \Xi_{\fs}\parallel; \fs\in \mathcal{V} \right\} \cdot |\theta|_{\infty}^{n-1}$. Compute

\[   \ell:={\rm{max}}\left\{\log_{|\theta|_{\infty}}D +1,\frac{nq}{q-1}+1  \right\}  .\]

\item[II-2] Let $f_{\fs}\in A[t]$ be the polynomial associated to $\Xi_{\fs}$ given in (\ref{E:fu}). Let $\left\{ a_{\fs}\right\}_{\fs\in \mathcal{V}}\subseteq \FF_{q}[t]$ be  {\it{parameters}}, and put $F:=\sum_{\fs\in \mathcal{V}}a_{\fs}f_{\fs}$.

\item[II-3] Let $c_{1},\ldots,c_{\ell}\in \FF_{q}[t]$ be {\it{parameters}} and put $\delta:=c_{1}\theta^{\ell-1}+\cdots+c_{\ell}\in \FF_{q}[t][\theta]$. Write $ F=d_{1}\theta^{\ell-1}+\cdots+d_{\ell}\in \FF_{q}[t][\theta] $
and note that the coefficients $d_{1},\ldots,d_{\ell}$ are $\FF_{q}[t]$-linear combinations of $\left\{a_{\fs} \right\}_{\fs\in\mathcal{V}}$.
\item[II-4] Comparing the coefficients of $\theta^{i}$ from the equation $\delta (t-\theta^{q})^{n}+F(t-\theta^{q})^{n}=\delta^{(1)}$, we obtain a system of linear equations in $c_{1},\ldots,c_{\ell}$ and $\left\{ a_{\fs}\right\}_{\fs\in \mathcal{V}}$ with coefficients in $\FF_{q}[t]$. Using Gaussian elimination we solve for solutions $c_{1},\ldots,c_{\ell}$ and $\left\{ a_{\fs}\right\}_{\fs\in \mathcal{V}}$, and particularly solve for rank $\tilde{r}_{n}$ of the solutions $[a_{\fs}]_{\fs\in \mathcal{V}}$, which is the number of independent $\FF_{q}[t]$-linear relations among $\left\{ \Xi_{\fs} \right\}_{\fs\in \mathcal{V}}$.

\item[II-5] Compute $r_{n}:=\lfloor\frac{n-1}{q-1} \rfloor - \tilde{r}_{n}$, which is the rank of the $\FF_{q}[t]$-module
\[\Sp_{\FF_{q}[t]}\left\{ \Xi_{\fs}\right\}_{\fs\in \mathcal{V}}   .\]
Compute $d_{n}:=n- \lfloor\frac{n-1}{q-1} \rfloor  +r_{n}$, which is the exact dimension we want by Theorem~\ref{T:MainThmIntroduction}.

\end{enumerate}

\subsection{Computational data} In this section, we list some data of implementing the algorithm above using Magma. We thank Yi-Hsuan Lin for providing
 the code. In what follows, \lq\lq Weight\lq\rq~means the weight $n$, and \lq\lq Dimension\rq\rq~means $d_{n}$ above, and \lq\lq $\FF_{p}$-linear\rq\rq~means the number of independent linear relations arising from (\ref{E:Chen}). When the weight $n$ is $A$-odd (ie., $(q-1)\nmid n$) the author does not know whether (\ref{E:Chen}) can produce a linear relation as $\zeta_{A}(r)\zeta_{A}(s)$ is a \lq\lq monomial\rq\rq~for one at least of $r, s$ being $A$-odd. So we let the position of $A$-odd weight be blank. \lq\lq Zeta-like\rq\rq~means the number of weight $n$ double zeta values $\zeta_{A}(\fs)$ for which $\zeta_{A}(\fs)/\zeta_{A}(n)\in k$. We list the computation data below only for $q=2$ and $q=3$, although we have run the program for other $q$ up to $11$ with weight up to $150$.

For $q=2$, we have:
$$
 \begin{tabular}{>{\centering}p{1.9cm}|>{\centering}p{0.4cm}|>{\centering}p{0.4cm}|>{\centering}p{0.4cm}|>{\centering}p{0.4cm}|>{\centering}p{0.4cm}|>{\centering}p{0.4cm}|>{\centering}p{0.4cm}|>{\centering}p{0.4cm}|>{\centering}p{0.4cm}|>{\centering}p{0.4cm}|>{\centering}p{0.4cm}|>{\centering}p{0.4cm}|>{\centering}p{0.4cm}|>{\centering\arraybackslash}p{0.4cm}}
 Weight&2&3&4&5&6&7&8&9&10&11&12&13&14&15\\
 \hline
 Dimension&1&2&2&3&3&3&3&4&4&4&4&5&5&5\\
 \hline
 $\FF_p$-linear&0&1&1&1&2&2&2&3&3&3&4&4&4&5\\
 \hline
 Zeta-like&1&1&1&0&0&2&1&0&0&0&0&0&0&2\\
 \end{tabular}
 $$

 $$
 \begin{tabular}{>{\centering}p{0.5cm}|>{\centering}p{0.5cm}|>{\centering}p{0.5cm}|>{\centering}p{0.5cm}|>{\centering}p{0.5cm}|>{\centering}p{0.5cm}|>{\centering}p{0.5cm}|>{\centering}p{0.5cm}|>{\centering}p{0.5cm}|>{\centering}p{0.5cm}|>{\centering}p{0.5cm}|>{\centering}p{0.5cm}|>{\centering}p{0.5cm}|>{\centering}p{0.5cm}|>{\centering\arraybackslash}p{0.5cm}}
 16&17&18&19&20&21&22&23&24&25&26&27&28&29&30\\
 \hline
 5&6&6&6&6&7&7&7&7&8&8&8&8&8&8\\
 \hline
 5&5&6&6&6&7&7&7&8&8&8&9&9&9&10\\
 \hline
 0&0&0&0&0&0&0&0&0&0&0&0&0&0&0\\
 \end{tabular}
 $$

 $$
 \begin{tabular}{>{\centering}p{0.5cm}|>{\centering}p{0.5cm}|>{\centering}p{0.5cm}|>{\centering}p{0.5cm}|>{\centering}p{0.5cm}|>{\centering}p{0.5cm}|>{\centering}p{0.5cm}|>{\centering}p{0.5cm}|>{\centering}p{0.5cm}|>{\centering}p{0.5cm}|>{\centering}p{0.5cm}|>{\centering}p{0.5cm}|>{\centering}p{0.5cm}|>{\centering}p{0.5cm}|>{\centering\arraybackslash}p{0.5cm}}
 31&32&33&34&35&36&37&38&39&40&41&42&43&44&45\\
 \hline
 8&8&9&9&9&9&10&10&10&10&11&11&11&11&11\\
 \hline
 10&10&11&11&11&12&12&12&13&13&13&14&14&14&15\\
 \hline
 2&0&0&0&0&0&0&0&0&0&0&0&0&0&0\\
 \end{tabular}
 $$

 $$
 \begin{tabular}{>{\centering}p{0.5cm}|>{\centering}p{0.5cm}|>{\centering}p{0.5cm}|>{\centering}p{0.5cm}|>{\centering}p{0.5cm}|>{\centering}p{0.5cm}|>{\centering}p{0.5cm}|>{\centering}p{0.5cm}|>{\centering}p{0.5cm}|>{\centering}p{0.5cm}|>{\centering}p{0.5cm}|>{\centering}p{0.5cm}|>{\centering}p{0.5cm}|>{\centering}p{0.5cm}|>{\centering\arraybackslash}p{0.5cm}}
 46&47&48&49&50&51&52&53&54&55&56&57&58&59&60\\
 \hline
 11&11&11&12&12&12&12&12&12&12&12&13&13&13&13\\
 \hline
 15&15&16&16&16&17&17&17&18&18&18&19&19&19&20\\
 \hline
 0&0&0&0&0&0&0&0&0&0&0&0&0&0&0\\
 \end{tabular}
 $$

 $$
 \begin{tabular}{>{\centering}p{0.5cm}|>{\centering}p{0.5cm}|>{\centering}p{0.5cm}|>{\centering}p{0.5cm}|>{\centering}p{0.5cm}|>{\centering}p{0.5cm}|>{\centering}p{0.5cm}|>{\centering}p{0.5cm}|>{\centering}p{0.5cm}|>{\centering}p{0.5cm}|>{\centering}p{0.5cm}|>{\centering}p{0.5cm}|>{\centering}p{0.5cm}|>{\centering}p{0.5cm}|>{\centering\arraybackslash}p{0.5cm}}
 61&62&63&64&65&66&67&68&69&70&71&72&73&74&75\\
 \hline
 13&13&13&13&14&14&14&14&15&15&15&15&16&16&16\\
 \hline
 20&20&21&21&21&22&22&22&23&23&23&24&24&24&25\\
 \hline
 0&0&2&0&0&0&0&0&0&0&0&0&0&0&0\\
 \end{tabular}
 $$

 $$
 \begin{tabular}{>{\centering}p{0.5cm}|>{\centering}p{0.5cm}|>{\centering}p{0.5cm}|>{\centering}p{0.5cm}|>{\centering}p{0.5cm}|>{\centering}p{0.5cm}|>{\centering}p{0.5cm}|>{\centering}p{0.5cm}|>{\centering}p{0.5cm}|>{\centering}p{0.5cm}|>{\centering}p{0.5cm}|>{\centering}p{0.5cm}|>{\centering}p{0.5cm}|>{\centering}p{0.5cm}|>{\centering\arraybackslash}p{0.5cm}}
 76&77&78&79&80&81&82&83&84&85&86&87&88&89&90\\
 \hline
 16&16&16&16&16&17&17&17&17&17&17&17&17&18&18\\
 \hline
 25&25&26&26&26&27&27&27&28&28&28&29&29&29&30\\
 \hline
 0&0&0&0&0&0&0&0&0&0&0&0&0&0&0\\
 \end{tabular}
 $$

 $$
 \begin{tabular}{>{\centering}p{0.5cm}|>{\centering}p{0.5cm}|>{\centering}p{0.5cm}|>{\centering}p{0.5cm}|>{\centering}p{0.5cm}|>{\centering}p{0.5cm}|>{\centering}p{0.5cm}|>{\centering}p{0.5cm}|>{\centering}p{0.5cm}|>{\centering}p{0.5cm}|>{\centering}p{0.5cm}|>{\centering}p{0.5cm}|>{\centering}p{0.5cm}|>{\centering}p{0.5cm}|>{\centering\arraybackslash}p{0.5cm}}
 91&92&93&94&95&96&97&98&99&100&101&102&103&104&105\\
 \hline
 18&18&18&18&18&18&19&19&19&19&19&19&19&19&20\\
 \hline
 30&30&31&31&31&32&32&32&33&33&33&34&34&34&35\\
 \hline
 0&0&0&0&0&0&0&0&0&0&0&0&0&0&0\\
 \end{tabular}
 $$

 $$
 \begin{tabular}{>{\centering}p{0.5cm}|>{\centering}p{0.5cm}|>{\centering}p{0.5cm}|>{\centering}p{0.5cm}|>{\centering}p{0.5cm}|>{\centering}p{0.5cm}|>{\centering}p{0.5cm}|>{\centering}p{0.5cm}|>{\centering}p{0.5cm}|>{\centering}p{0.5cm}|>{\centering}p{0.5cm}|>{\centering}p{0.5cm}|>{\centering}p{0.5cm}|>{\centering}p{0.5cm}|>{\centering\arraybackslash}p{0.5cm}}
 106&107&108&109&110&111&112&113&114&115&116&117&118&119&120\\
 \hline
 20&20&20&20&20&20&20&21&21&21&21&21&21&21&21\\
 \hline
 35&35&36&36&36&37&37&37&38&38&38&39&39&39&40\\
 \hline
 0&0&0&0&0&0&0&0&0&0&0&0&0&0&0\\
 \end{tabular}
 $$

For $q=3$, we have:
$$
\begin{tabular}{>{\centering}p{1.9cm}|>{\centering}p{0.4cm}|>{\centering}p{0.4cm}|>{\centering}p{0.4cm}|>{\centering}p{0.4cm}|>{\centering}p{0.4cm}|>{\centering}p{0.4cm}|>{\centering}p{0.4cm}|>{\centering}p{0.4cm}|>{\centering}p{0.4cm}|>{\centering}p{0.4cm}|>{\centering}p{0.4cm}|>{\centering}p{0.4cm}|>{\centering}p{0.4cm}|>{\centering\arraybackslash}p{0.4cm}}
Weight&3&4&5&6&7&8&9&10&11&12&13&14&15&16\\
\hline
Dimension&3&4&5&5&7&7&8&9&10&10&12&12&13&13\\
\hline
$\FF_p$-linear& &0& &0& &1& &1& &1& &1& &2\\
\hline
Zeta-like&1&0&1&1&1&1&1&0&0&0&0&0&0&0\\
\end{tabular}
$$

$$
\begin{tabular}{>{\centering}p{0.5cm}|>{\centering}p{0.5cm}|>{\centering}p{0.5cm}|>{\centering}p{0.5cm}|>{\centering}p{0.5cm}|>{\centering}p{0.5cm}|>{\centering}p{0.5cm}|>{\centering}p{0.5cm}|>{\centering}p{0.5cm}|>{\centering}p{0.5cm}|>{\centering}p{0.5cm}|>{\centering}p{0.5cm}|>{\centering}p{0.5cm}|>{\centering}p{0.5cm}|>{\centering\arraybackslash}p{0.5cm}}
17&18&19&20&21&22&23&24&25&26&27&28&29&30&31\\
\hline
14&14&16&16&17&17&18&18&19&19&20&21&22&22&24\\
\hline
&2& &2& &2& &3& &3& &3& &3& \\
\hline
2&0&0&0&0&0&2&0&3&2&2&0&0&0&0\\
\end{tabular}
$$

$$
\begin{tabular}{>{\centering}p{0.5cm}|>{\centering}p{0.5cm}|>{\centering}p{0.5cm}|>{\centering}p{0.5cm}|>{\centering}p{0.5cm}|>{\centering}p{0.5cm}|>{\centering}p{0.5cm}|>{\centering}p{0.5cm}|>{\centering}p{0.5cm}|>{\centering}p{0.5cm}|>{\centering}p{0.5cm}|>{\centering}p{0.5cm}|>{\centering}p{0.5cm}|>{\centering}p{0.5cm}|>{\centering\arraybackslash}p{0.5cm}}
32&33&34&35&36&37&38&39&40&41&42&43&44&45&46\\
\hline
24&25&25&26&26&28&28&29&29&30&30&31&31&32&33\\
\hline
4& &4& &4& &4& &5& &5& &5& &5\\
\hline
0&0&0&0&0&0&0&0&0&0&0&0&0&0&0\\
\end{tabular}
$$

$$
\begin{tabular}{>{\centering}p{0.5cm}|>{\centering}p{0.5cm}|>{\centering}p{0.5cm}|>{\centering}p{0.5cm}|>{\centering}p{0.5cm}|>{\centering}p{0.5cm}|>{\centering}p{0.5cm}|>{\centering}p{0.5cm}|>{\centering}p{0.5cm}|>{\centering}p{0.5cm}|>{\centering}p{0.5cm}|>{\centering}p{0.5cm}|>{\centering}p{0.5cm}|>{\centering}p{0.5cm}|>{\centering\arraybackslash}p{0.5cm}}
47&48&49&50&51&52&53&54&55&56&57&58&59&60&61\\
\hline
34&34&35&35&36&36&37&37&39&39&40&40&41&41&42\\
\hline
&6& &6& &6& &6& &7& &7& &7& \\
\hline
0&0&0&0&0&0&2&0&0&0&0&0&0&0&0\\
\end{tabular}
$$

$$
\begin{tabular}{>{\centering}p{0.5cm}|>{\centering}p{0.5cm}|>{\centering}p{0.5cm}|>{\centering}p{0.5cm}|>{\centering}p{0.5cm}|>{\centering}p{0.5cm}|>{\centering}p{0.5cm}|>{\centering}p{0.5cm}|>{\centering}p{0.5cm}|>{\centering}p{0.5cm}|>{\centering}p{0.5cm}|>{\centering}p{0.5cm}|>{\centering}p{0.5cm}|>{\centering}p{0.5cm}|>{\centering\arraybackslash}p{0.5cm}}
62&63&64&65&66&67&68&69&70&71&72&73&74&75&76\\
\hline
42&43&44&45&45&46&46&47&47&48&48&50&50&51&51\\
\hline
7& &8& &8& &8& &8& &9& &9& &9\\
\hline
0&0&0&0&0&0&0&0&0&2&0&0&0&0&0\\
\end{tabular}
$$

$$
\begin{tabular}{>{\centering}p{0.5cm}|>{\centering}p{0.5cm}|>{\centering}p{0.5cm}|>{\centering}p{0.5cm}|>{\centering}p{0.5cm}|>{\centering}p{0.5cm}|>{\centering}p{0.5cm}|>{\centering}p{0.5cm}|>{\centering}p{0.5cm}|>{\centering}p{0.5cm}|>{\centering}p{0.5cm}|>{\centering}p{0.5cm}|>{\centering}p{0.5cm}|>{\centering}p{0.5cm}|>{\centering\arraybackslash}p{0.5cm}}
77&78&79&80&81&82&83&84&85&86&87&88&89&90&91\\
\hline
52&52&53&53&54&55&56&56&58&58&59&59&60&60&62\\
\hline
&9& &10& &10& &10& &10& &11& &11& \\
\hline
4&0&3&2&0&0&0&0&0&0&0&0&0&0&0\\
\end{tabular}
$$

$$
\begin{tabular}{>{\centering}p{0.5cm}|>{\centering}p{0.5cm}|>{\centering}p{0.5cm}|>{\centering}p{0.5cm}|>{\centering}p{0.5cm}|>{\centering}p{0.5cm}|>{\centering}p{0.5cm}|>{\centering}p{0.5cm}|>{\centering}p{0.5cm}|>{\centering}p{0.5cm}|>{\centering}p{0.5cm}|>{\centering}p{0.5cm}|>{\centering}p{0.5cm}|>{\centering}p{0.5cm}|>{\centering\arraybackslash}p{0.5cm}}
92&93&94&95&96&97&98&99&100&101&102&103&104&105&106\\
\hline
62&63&63&64&64&65&65&66&67&68&68&69&69&70&70\\
\hline
11& &11& &12& &12& &12& &12& &13& &13\\
\hline
0&0&0&0&0&0&0&0&0&0&0&0&0&0&0\\
\end{tabular}
$$

$$
\begin{tabular}{>{\centering}p{0.5cm}|>{\centering}p{0.5cm}|>{\centering}p{0.5cm}|>{\centering}p{0.5cm}|>{\centering}p{0.5cm}|>{\centering}p{0.5cm}|>{\centering}p{0.5cm}|>{\centering}p{0.5cm}|>{\centering}p{0.5cm}|>{\centering}p{0.5cm}|>{\centering}p{0.5cm}|>{\centering}p{0.5cm}|>{\centering}p{0.5cm}|>{\centering}p{0.5cm}|>{\centering\arraybackslash}p{0.5cm}}
107&108&109&110&111&112&113&114&115&116&117&118&119&120&121\\
\hline
71&71&73&73&74&74&75&75&76&76&77&78&79&79&80\\
\hline
&13& &13& &14& &14& &14& &14& &15& \\
\hline
0&0&0&0&0&0&0&0&0&0&0&0&0&0&0\\
\end{tabular}
$$

$$
\begin{tabular}{>{\centering}p{0.5cm}|>{\centering}p{0.5cm}|>{\centering}p{0.5cm}|>{\centering}p{0.5cm}|>{\centering}p{0.5cm}|>{\centering}p{0.5cm}|>{\centering}p{0.5cm}|>{\centering}p{0.5cm}|>{\centering}p{0.5cm}|>{\centering}p{0.5cm}|>{\centering}p{0.5cm}|>{\centering}p{0.5cm}|>{\centering}p{0.5cm}|>{\centering}p{0.5cm}|>{\centering\arraybackslash}p{0.5cm}}
122&123&124&125&126&127&128&129&130&131&132&133&134&135&136\\
\hline
80&81&81&82&82&84&84&85&85&86&86&87&87&88&89\\
\hline
15& &15& &15& &16& &16& &16& &16& &17\\
\hline
0&0&0&0&0&0&0&0&0&0&0&0&0&0&0\\
\end{tabular}
$$

$$
\begin{tabular}{>{\centering}p{0.5cm}|>{\centering}p{0.5cm}|>{\centering}p{0.5cm}|>{\centering}p{0.5cm}|>{\centering}p{0.5cm}|>{\centering}p{0.5cm}|>{\centering}p{0.5cm}|>{\centering}p{0.5cm}|>{\centering}p{0.5cm}|>{\centering}p{0.5cm}|>{\centering}p{0.5cm}|>{\centering}p{0.5cm}|>{\centering}p{0.5cm}|>{\centering\arraybackslash}p{0.5cm}}
137&138&139&140&141&142&143&144&145&146&147&148&149&150\\
\hline
90&90&91&91&92&92&93&93&95&95&96&96&97&97\\
\hline
&17& &17& &17& &18& &18& &18& &18\\
\hline
0&0&0&0&0&0&0&0&0&0&0&0&0&0\\
\end{tabular}
$$

\bibliographystyle{alpha}

\end{document}